\documentclass{gen-j-l}

\newtheorem{theorem}{Theorem}[section]
\newtheorem{lemma}[theorem]{Lemma}
\newtheorem{corollary}[theorem]{Corollary}
\newtheorem{proposition}[theorem]{Proposition}

\theoremstyle{definition}
\newtheorem{definition}[theorem]{Definition}
\newtheorem{assumption}[theorem]{Assumption}

\theoremstyle{remark}
\newtheorem{remark}[theorem]{Remark}

\numberwithin{equation}{section}

\usepackage{bm}
\usepackage{algorithm}
\usepackage{algpseudocode}
\usepackage{xcolor}
\usepackage{mathtools}
\usepackage{bbm}
\usepackage{placeins}
\usepackage{hyperref}
\def\red#1{#1}
\def\blue#1{#1}
\DeclareMathOperator*{\argmin}{arg\,min}
\DeclareMathOperator\supp{supp}

\begin{document}

\title{A Tikhonov regularization based algorithm for scattered data with random noise}


%    author one information
\author{Jiantang Zhang}
\address{School of Mathematical Sciences, Fudan University, Shanghai 200433, China}
\email{15110180015@fudan.edu.cn}

%    author two information
\author{Jin Cheng}
\address{School of Mathematical Sciences, Fudan University, Shanghai 200433, China}
\email{jcheng@fudan.edu.cn}

\author{Min Zhong}
\address{Nanjing Center for Applied Mathematics, 211135, Nanjing, Jiangsu Province}
\email{min.zhong@seu.edu.cn}

%    \subjclass is required by all journals except JAG.
% \subjclass[2020]{Primary }

% \date{}

% \dedicatory{}

%    The "communicated by" line appears only in PROC and JAG.
%\commby{}

%    Abstract is required.
\begin{abstract}
With the rapid growth of data, how to extract effective information from data is one of the most fundamental problems.
In this paper, based on Tikhonov regularization, we propose an effective method for reconstructing the function and its derivative from scattered data with random noise.
Since the noise level is not assumed small, we will use the amount of data for reducing the random error, and use a relatively small number of knots for interpolation.
An indicator function for our algorithm is constructed.
It indicates where the numerical results are good or may not be good.
The corresponding error estimates are obtained.
We show how to choose the number of interpolation knots in the reconstruction process for balancing the random errors and interpolation errors.
Numerical examples show the effectiveness and rapidity of our method.
It should be remarked that the algorithm in this paper can be used for on-line data.
% We develop a big data processing technique that reconstructs a function and its derivative from sample values with random noise.
% The solution space of function reconstruction is a finite dimensional vector space whose dimension is relatively small, which reduces computational burden.
% With a given dataset, we design an indicator function which shows the effective regions where reconstruction results are reliable.
% We also give asymptotic convergence rates, and provide numerical examples.
\end{abstract}

\maketitle


\section{Introduction} % (fold)
\label{sec:introduction}

Suppose that $f(x)$ is a function defined on $[0, 1]$.
We consider the following problem: for a positive integer $N$, given observation points $\{x_i\}_{i=1}^{N} \subseteq [0, 1]$ and corresponding noisy samples $y_i$ of function values $f(x_i)$ which satisfy
\begin{equation}\label{eq:problem-formulation}
	y_i = f(x_i) + \eta_i,\quad 1 \leq i \leq N,
\end{equation}
where \red{the} observation noise $\eta_i, 1 \leq i \leq N$ are uncorrelated \red{random variables with mean zero and variance $\sigma^2$}, that is,
\begin{equation*}
	\mathbb{E} [\eta_i] = 0,\quad \mathbb{E} [\eta_i \eta_j] = \sigma^2\delta_{ij},\quad 1 \leq i, j \leq N,
\end{equation*}
\red{in which the} $\mathbb{E}[\cdot]$ \red{stands} for expectation and $\delta_{ij}$ stands for Kronecker symbol. \red{We are willing to} construct a function $f_N(x)$ such that the derivative of $f_N(x)$ approximates the derivative of $f(x)$. \red{Such a problem is called numerical differentiation, which is widely applied in various problems \cite{Deans.2007,Gorenflo.2006,Hanke.2001}. Numerical differentiation is a classical ill-posed problem in the sense of unstable dependence of solutions on small perturbations of data. Therefore, regularization methods should be taken into consideration. There have been plenty of regularization methods for treating such ill-posed problems in one dimension or higher dimensions, see \cite{Hanke.2001,YBWang.2002,Cheng.2007,Lu.2006,Lu.2006b,Hu.2012} and references therein. However, those traditional approaches are based on accurate information of the noise bound $\delta$ or a good prediction of it, therefore not suitable for randomly distributed noise,} \blue{as the noise bound cannot be effectively controlled.} In the field of statistics, \cite{Wahba.1975,Craven.1978,Ragozin.1983,Scott.1989} considered similar data smoothing problems with independent or uncorrelated random noise. In these works, the data is assumed to be quasi-unform, when the sample size tends to infinity, the construction results converges to sought solutions. \blue{However, the memory usage also increases when sample size gets large\cite{Ahnert.2007}, and becomes a burden when dealing with very large amount of data.}

\red{Compared with classical numerical differentiation, there are two main difficulties in our problem. First, the sample size could be very large, and due to inevitable measurement errors in the observations, there are inevitable randomly distributed noise whose variance cannot be very small. How to take advantage from the large amount of data to reduce random noise and improve the accuracy is one of the most fundamental problems. In addition, an appropriate regularization parameter selection rule should be carefully discussed, which should not rely on the noise bound. Second, the position of observation points may not be quasi-uniform, or even randomly designed, how to determine reliable regions and provide asymptotic convergence property should be taken into consideration as well. }

In order to solve these barriers and difficulties, we propose a \red{statistical} Tikhonov regularization algorithm, which makes good use of big data at a relatively low computational cost.
Inspired by penalized splines in statistics \cite{OSullivan.1986,Eilers.1996,Wand.2008,Claeskens.2009}, we fix a set of equidistant interpolation knots, and search the regularized solution in a \red{projected} space. At the same time, we prove that only a small number of interpolation knots are \red{necessary} to achieve good reconstruction accuracy, thus the computational cost, \blue{especially memory usage,} is effectively reduced. We also propose a prior choice rule for regularization parameter, which gives optimal convergence rates.

To better deal with data that are unevenly spaced, we introduce the histogram of observation points as an indicator function to show reliable regions \red{in which} the results are supposed to be accurate.
In this way, it is unnecessary to impose additional {\it a prior} conditions. On the other hand, if observation points are randomly designed, we are able to provide asymptotic convergence rates in probability as well.

The rest of this paper is organized as follows.
In Section \ref{sec:set_up}, we formulate the problem and propose an on-line reconstruction algorithm, the {\it a prior} choice rule for regularization parameter is discussed.
In Section \ref{sec:error_analysis}, we give error analysis in confidence interval and convergence rates in probability.
In Section \ref{sec:numerical_examples}, we provide several numerical examples. The conclusions are contained in Section \ref{sec:conclusion}.

\section{Formulation of the problem} % (fold)
\label{sec:set_up}

% Suppose that $I = [0, 1]$, and $f(x)$ is an unknown function that belongs to the Sobolev space $W^{2, 2}(I)$.
% We consider the following problem: given some observation points $\{x_i\}_{i=1}^{N} \subset I$ and corresponding noisy samples $y_i$ of the function values $f(x_i)$, which satisfy
% \begin{equation*}
% 	y_i = f(x_i) + \eta_i,\quad i = 1, 2, \cdots N,
% \end{equation*}
% we want to construct a function $f_N(x)$ such that $f_N(x)$ approximates $f(x)$ and $f'_N(x)$ approximates $f'(x)$.
% Here, we suppose that observation errors $\{\eta_i\}_{i = 1}^{N}$ are uncorrelated, zero mean random variables with a common variance of $\sigma^2$, that is,
% \begin{equation}
% 	\mathbb{E} [\eta_i] = 0,\quad \mathbb{E} [\eta_i \eta_j] = \delta_{ij} \sigma^2,\quad 1 \leq i, j \leq N.
% \end{equation}

% For a typical large dataset, the size $N$ can be quite large, while the variance $\sigma^2$ cannot be arbitrarily small due to limitations on measurements.
% Therefore, we desire a low-cost algorithm that reduces reconstruction error by making good use of the large amount of samples.
%In this section, we construct the regularized solution and propose the reconstruction algorithm.
%In the following text, we denote
%\begin{equation}
%	\begin{split}
%	\bm{x}_N &= (x_1, x_2, \cdots, x_N)^\top, \\
%	\bm{y}_N &= (y_1, y_2, \cdots, y_N)^\top, \\
%	\bm{\eta}_N &= (\eta_1, \eta_2, \cdots, \eta_N)^\top,
%	\end{split}
%\end{equation}
%where $x_i, y_i, \eta_i, 1 \leq i \leq N$ are defined in \eqref{eq:problem-formulation}.

\subsection{Tikhonov functional and regularized solution} % (fold)
\label{sub:regularization_in_a_vector_space_of_a_proper_dimension}
In this part, we construct the regularized solution and propose the reconstruction algorithm. \blue{First, we} define a finite dimensional linear space $V_M$ in which the regularized solution \red{is established}.

\begin{definition}[Definition of $V_M$]\label{definition:vm}
	Let $M$ be a positive integer and mesh size $d = M^{-1}$.
	Define equidistant knots $\{p_j\}_{j \in \mathbb{Z}}$ by
	\begin{equation}\label{eq:pj-definition}
		p_j = j d,\quad j \in \mathbb{Z},
	\end{equation}
	$V_M$ is the vector space of all cubic spline functions with knots $\{p_j\}_{j=0}^{M}$.
\end{definition}

\begin{remark}
	Suppose that $f(x) \in W^{2, 2}(0, 1)$, the cubic spline functions can \red{provide} good approximations for $f(x)$ and its derivative.
\end{remark}

Here, $M$ \red{represents the dimension of the regularized solution space, which significantly affects both} the computational cost and the approximation accuracy, thus should be chosen appropriately.
We will discuss the choice of $M$ in following Remark \ref{remark:choice-of-m}.
% Tikhonov regularization is widely applied to ill-posed problems.
% In data smoothing, Tikhonov regularization also has many applications \cite{Hanke.2001} \cite{Cheng.2007} \cite{Wahba.1975} \cite{Scott.1989}.
% In these approaches, if the regularization term is $L^2$ norm of second order derivative, the solution will be a cubic spline function whose knots are identical to observation points $\{x_i\}_{i = 1}^{N}$.
% When sample size $N$ is large, these approaches have too much computational burden.

% One innovation point of our approach is to find the minimizer of Tikhonov functional in a finite dimensional linear space that is chosen according to requirements for accuracy rather than observation points.
% In this way, the number of degrees of freedom to be determined can be reduced to a great extent without compromising on the result accuracy, thus computational burden is alleviated.
% In fact, our approach gives reconstruction results that are less susceptible to random noise thanks to fewer degrees of freedom.
% At the end of this section, we provide an on-line algorithm that computes this regularized solution based on sequential data.

%Next, we introduce the Tikhonov functional.
%Since we want to approximate the first-order derivative, we use the $L^2$ norm of second-order derivative as regularization term.
For $N\geq 2$, denote the noisy sample $\bm{y}_N = (y_1, y_2, \cdots, y_N)^\top$ and define the following Tikhonov functional
\begin{equation}\label{eq:tikhonov-functional}
	J(g; \alpha_N, \bm{y}_N) = \frac{1}{N} \sum_{i=1}^{N} \big( g(x_i) - y_i \big)^2 + \alpha_N \Vert g'' \Vert_{L^2(0, 1)}^2,
\end{equation}
in which $g(x) \in W^{2, 2}(0, 1)$, and $\alpha_N > 0$ is a regularization parameter. Consider the minimization problem
\begin{equation}\label{minimization}
	f_N = \argmin_{g \in V_M} J(g; \alpha_N, \bm{y}_N),
\end{equation}
\red{we use $f_N$ as the approximated solution of numerical differentiation problem.}

The cubic B-splines \cite{Micula.1999} can be utilized to construct a basis of linear subspace $V_M$, see Appendix \ref{sec:appendix-a-b-splines} for its $(M+3)$ basis $\{\psi_j\}_{j = -1}^{M+1}$. 
% In addition, for arbitrary $\bm{\lambda}_N = (\lambda_1, \lambda_2, \cdots, \lambda_N)^\top$, it can be proved that there exists an linear isomorphism from $\mathbb{R}^{M+3}$ to $V_M$ by
\blue{As a notation, for a column vector $\bm{\lambda}_N = (\lambda_{-1}, \lambda_0, \lambda_1, \cdots, \lambda_{M+1})^\top \in \mathbb{R}^{M+3}$, we define a linear isomorphism from $\mathbb{R}^{M+3}$ to $V_M$ by}

\begin{equation*}
	\begin{split}
		\Phi : \mathbb{R}^{M+3} &\to V_M, \\
		\bm{\lambda} &\mapsto \Phi[\bm{\lambda}] = \sum_{j=-1}^{M+1} \lambda_j \psi_j.
	\end{split}
\end{equation*}
We also introduce the following proposition.

\begin{proposition}
    \blue{For arbitrary $\bm{\lambda} \in \mathbb{R}^{M+3}$, the value of function $\Phi[\bm{\lambda}]$ at $x \in [0, 1]$ can be expressed as
    \begin{equation}\label{eq:hx-property}
        \Phi[\bm{\lambda}](x) = H_x \bm{\lambda},
    \end{equation}
    in which $(M+3)$-dimensional row vector $H_x = ( \psi_{-1}(x), \psi_0(x), \cdots, \psi_M(x), \psi_{M+1}(x) )$.
    The $L^2$ norm of the second order derivative of $\Phi[\bm{\lambda}]$ can be expressed as
    \begin{equation}\label{eq:precision-matrix-property}
        \Vert \Phi[\bm{\lambda}]''(x) \Vert_{L^2(0, 1)}^2 = \bm{\lambda}^\top P \bm{\lambda},
    \end{equation}
    in which $P \in \mathbb{R}^{(M+3) \times (M+3)}$ is defined by
    \begin{equation*}
        P = (p_{ij})_{i, j = -1}^{M+1},\quad p_{ij} = \int_0^1 \psi_i''(x) \psi_j''(x) \,d{x}.
    \end{equation*}}
\end{proposition}

% \begin{proposition}
% For arbitrary $x\in[0,1]$, define a $(M+3)$-dimensional row vector $H_x$ by
% 	\begin{equation*}
% 		H_x = \big( \psi_{-1}(x), \psi_0(x), \cdots, \psi_M(x), \psi_{M+1}(x) \big).
% 	\end{equation*}
% Then, for any $\bm{\lambda} \in \mathbb{R}^{(M+3)}$, there holds
% 	\begin{equation}\label{eq:hx-property}
% 		\Phi[\bm{\lambda}](x) = H_x \bm{\lambda},
% 	\end{equation}
% and
% 	\begin{equation}\label{eq:precision-matrix-property}
% 		\bm{\lambda}^\top P \bm{\lambda} = \Vert \Phi[\bm{\lambda}]''(x) \Vert_{L^2(0, 1)}^2,
% 	\end{equation}
% in which the matrix $P \in \mathbb{R}^{(M+3) \times (M+3)}$ is defined by
% 	\begin{equation*}
% 		P = (p_{ij})_{i, j = -1}^{M+1},\quad p_{ij} = \int_0^1 \psi_i''(x) \psi_j''(x) \,d{x}.
% 	\end{equation*}
% \end{proposition}

\begin{remark}
	Since
	\begin{equation*}
		\supp \psi_j = [p_{j-2}, p_{j+2}],\quad j = -1, 0, \cdots, M+1,
	\end{equation*}
	the matrix $P$ is a band matrix with a bandwidth of $3$.
	On the other hand, the row vector $H_x$ has at most $4$ nonzero elements, and their positions are continuous.
	Thus, $H_x^\top H_x$ is also a band matrix with a maximum bandwidth of $3$.
\end{remark}

Next, we illustrate how to minimize the Tikhonov functional \eqref{eq:tikhonov-functional} in the subspase $V_M$. Define a matrix $H_N \in \mathbb{R}^{N \times (M+3)}$ by
\begin{equation*}
	H_N = \begin{pmatrix}
		H_{x_1} \\ \vdots \\ H_{x_N}
	\end{pmatrix},
\end{equation*}
then, for arbitrary $\bm{\lambda} \in \mathbb{R}^{(M+3)}$, the Tikhonov functional \eqref{eq:tikhonov-functional} \red{can be rewritten in a matrix form},
\begin{equation}\label{eq:tikhonov-functional-vector-form}
	J(\Phi[\bm{\lambda}]; \alpha_N, \bm{y}_N) = \frac{1}{N} (H_N \bm{\lambda} - \bm{y}_N)^\top (H_N \bm{\lambda} - \bm{y}_N) + \alpha_N \bm{\lambda}^\top P \bm{\lambda}.
\end{equation}

\begin{theorem}\label{thm:unique-minimizer-of-tikhonov-functional-in-vm}
   Suppose $N \geq 2$, and the observation points $\{x_i\}_{i=1}^{N}$ are not identical. Then, the Tikhonov minimization problem \eqref{minimization} has a unique minimizer
	\begin{equation*}
		f_N = \Phi[\bm{\lambda}_N].
	\end{equation*}
The coefficients $\bm{\lambda}_N \in \mathbb{R}^{(M+3)}$ can be solved from the linear system
	\begin{equation}\label{eq:lambda-n-definition}
		\Big( \alpha_N P + \frac{1}{N} H_N^\top H_N \Big) \bm{\lambda}_N = \frac{1}{N} H_N^\top \bm{y}_N.
	\end{equation}
\end{theorem}

\begin{proof}
	Since \eqref{eq:tikhonov-functional-vector-form} is a quadratic form with respect to $\bm{\lambda}$, it can be written as
	\begin{equation}\label{eq:tikhonov-functional-quadratic-form}
		J(\Phi[\bm{\lambda}]; \alpha_N, \bm{y}_N) = \frac{1}{2} (\bm{\lambda} - \bm{\lambda}_N)^\top A (\bm{\lambda} - \bm{\lambda}_N) + c,
	\end{equation}
	where
	\begin{equation*}
		A = \frac{\partial^2}{\partial \bm{\lambda}^2} J(\Phi[\bm{\lambda}]; \alpha_N, \bm{y}_N) = \frac{2}{N} H_N^\top H_N + 2 \alpha_N P,
	\end{equation*}
	\blue{and $c \in \mathbb{R}$ is independent of $\bm{\lambda}$.}
	Since $P$ and $H_N^\top H_N$ are both positive semidefinite and $\{x_i\}_{i=1}^{N}$ are not all identical, it is easy to conclude $A$ is positive definite. \blue{Since the derivative $\frac{\partial}{\partial \bm{\lambda}} J(\Phi[\bm{\lambda}]; \alpha_N, \bm{y}_N)$ takes $\bm{0}$ only at $\bm{\lambda} = \bm{\lambda}_N$, it follows that}
	\begin{equation*}
		\Big( \alpha_N P + \frac{1}{N} H_N^\top H_N \Big) \bm{\lambda}_N = \frac{1}{N} H_N^\top \bm{y}_N.
	\end{equation*}
\end{proof}

\subsection{Algorithm} % (fold)
\label{sub:algorithm_description}
%We give an algorithm in pseudo code that constructs the approximant $f_N$.
%In the following algorithm, we choose the regularization parameter $\alpha_N$ in prior by
%\begin{equation}\label{eq:prior-choice-of-regularization-parameter}
%	\alpha_N = \frac{\sigma^2 M}{N} + d^4,
%\end{equation}
%where $d = M^{-1}$ is the mesh size.
%The reason of this prior choice is to control both random error and interpolation error.
%Corresponding error bounds and convergence rates are given in Section \ref{sec:error_analysis}.
\ 

\begin{algorithm}[H]
\caption{\red{The Online Tikhonov regularization for scattered data with random noise}}
\label{algorithm:basic}
\begin{algorithmic}[1]
  \Require The number of knots $M$, mesh size $d=1/M$, the number of sample $N$, the observation data $\{(x_i, y_i)\}_{i = 1}^{N}$ and the variance $\sigma^2$;
  \Ensure The approximate solution $f_N(x) \in V_M$.
  \State Initialize $A_0 = 0 \in \mathbb{R}^{(M+3) \times (M+3)}$;
  \State Initialize $\bm{b}_0 = \bm{0} \in \mathbb{R}^{(M+3)}$
  \State Generate the matrix $P \in \mathbb{R}^{(M+3) \times (M+3)}$, where
  \begin{equation*}
  P = (p_{ij})_{i,j=-1}^{M+1},\quad p_{ij} = \int_0^1 \psi_i''(x) \psi_j''(x) \,d{x};
  \end{equation*}
  \For {$i \gets 1, 2, \cdots, N$}
    \State Generate row vector $H_{x_i} = \big( \psi_{-1}(x_i), \psi_0(x_i), \cdots, \psi_M(x_i), \psi_{M+1}(x_i) \big)$;
    \State Update $A_i$ by $A_i \gets \frac{i-1}{i} A_{i-1} + \frac{1}{i} H_{x_i}^\top H_{x_i}$;
  	\State Update $\bm{b}_i$ by $\bm{b}_i \gets \frac{i-1}{i} \bm{b}_{i-1} + \frac{1}{i} H_{x_i}^\top y_i$;
  \EndFor
  \State Choose $\alpha_N = M \sigma^2 / N + d^4$, and solve linear system
  	\begin{equation*}
	( \alpha_N P +  A_N ) \bm{\lambda}_{N} =  \bm{b}_N
	\end{equation*}
	for $\bm{\lambda}_N \in \mathbb{R}^{(M+3)}$;
  \State Give function reconstruction result $f_N(x)$ by $f_N = \Phi[\bm{\lambda}_N]$.
\end{algorithmic}
\end{algorithm}

% (1) flops; (2) online manner; (3) when to enlarge M
\begin{remark}
Since the row vector $H_x$ has at most $4$ nonzero elements and the matrix $\alpha_N P +  A_N$ has a bandwidth of $3$, thus the \red{computational complexity }at Line $5$, $6$ and $7$ are $\mathcal{O}(1)$,
and the computational complexity at Line $10$ is $\mathcal{O}(M)$. In addition, the total data storage of this algorithm is $\mathcal{O}(M)$.

The algorithm supports an on-line update if $N$ increases.
When new data are considered, one may continue to run the algorithm from Line $5$ to Line $10$ without reprocessing old data. When $N$ becomes \blue{so large that} $M \sigma^2 / N < d^4$, one needs to increase $M$ and restart the algorithm \blue{to further improve accuracy}.
\end{remark}
\begin{remark}
\blue{The {\it a prior} parameter choice strategy of
\begin{equation*}
    \alpha_N = \frac{M \sigma^2}{N} + d^4
\end{equation*}
at Line 9 is a balance between stability and accuracy. Corresponding theoretical analysis are discussed in the next section.}
\end{remark}

%\begin{remark}[The choice of $M$]\label{remark:choice-of-m}
%	For the value of $M$, on the one hand, if $N$ is known in advance, according to Theorem \ref{theorem:convergence-random-obs-points}, one may choose $M \sim N^{1/5}$ for optimal convergence rates.
%	On the other hand, if $N$ is unknown, one may choose $M$ according to requirements for accuracy, since the reconstruction accuracy is roughly the same as the cubic spline interpolation with knots $\{p_j\}_{j=0}^{M}$ when $N$ is sufficiently large.
%\end{remark}

% subsection algorithm_description (end)

% section problem_formulation (end)

\section{Theoretical analysis} % (fold)
\label{sec:error_analysis}
\subsection{The indicator function and preliminary lemmas} % (fold)
\label{sub:preliminary_lemmas}

Since the observation points may not be quasi-uniform, the distribution of these points especially affects the approximation accuracy. For example, the approximated solution at places with fewer observation points is likely to be less accurate. Therefore, the histogram of observation points \red{will be introduced} to indicate their distribution and show reliable intervals of reconstruction.

\begin{definition}\label{definition:indicator-function}
   \red{Divide} $[0, 1]$ into $M$ subintervals $I_j, 1 \leq j \leq M$, that is,
	\begin{equation}\label{eq:ij-definition}
		\begin{split}
		I_1 &= [0, d], \\
		I_j &= (jd - d, jd].
		\end{split}
	\end{equation}
	For $j = 1, 2, \cdots, M$, denoting by $N_j$ the number of observation points which belong to $I_j$, indicator function $\rho_N(x)$ is \red{defined as}
	\begin{equation}\label{eq:rho-n-definition}
		\rho_N(x) = \begin{cases}
			\rho_{N, j} = N_j / N d, &\text{ if }x \in I_j, \\
			0, &\text{ if } x \notin [0, 1].
		\end{cases}
	\end{equation}
	% In the following text, we suppose that the upper bound of $\rho_N(x)$ satisfies
	% \begin{equation}
	% 	\sup_{x \in [0, 1]} \rho_N(x) \leq \beta_N.
	% \end{equation}
\end{definition}
From the definition,
	\begin{equation*}
		I_1 \cup I_2 \cup \cdots \cup I_M = [0, 1]
	\end{equation*}
	and $I_j, 1 \leq j \leq M$ do not intersect with each other. \red{It is obvious that,}
	\begin{equation*}
		N_1 + N_2 + \cdots N_M = N
	\end{equation*}
	and
	\begin{equation*}
		\int_0^1 \rho_N(x) \,d{x} = 1.
	\end{equation*}

\red{Utilizing the indicator function, the following preliminary lemmas are necessarily be provided, which will be the foundation of formal theoretical analysis.} The key result is Lemma \ref{lemma:estimation-longer-interval}, which bounds the $L^2$ norm of a function by its mean squared value at observation points and its second order derivative.
\begin{lemma}\label{lemma:point-l2-estimation}
\red{For arbitrary} $u(x) \in C^1[a, b]$ and $x_0 \in [a, b]$, \red{the $L^2$ norm of $u$ and the squared value of $u$ at $x_0$ can be estimated as}
	\begin{equation*}
	\begin{split}
		\Vert u \Vert_{L^2(a, b)}^2 &\leq 2 \Big( (b-a) u^2(x_0) + (b-a)^2 \Vert u' \Vert_{L^2(a, b)}^2 \Big), \\
		(b-a) u^2(x_0) &\leq 2 \Vert u \Vert_{L^2(a, b)}^2 + 2 (b-a)^2 \Vert u' \Vert_{L^2(a, b)}^2.
	\end{split}
	\end{equation*}
\end{lemma}

\begin{proof}
\red{For arbitrary} $x \in [a, b]$, we have
\begin{equation*}
	u(x) = u(x_0) + \int_{x_0}^{x} u'(s) \,d{s}.
\end{equation*}
Taking squares on both sides gives that
\begin{equation}\label{eq:point-l2-estimation-tmp}
	u^2(x) \leq 2 u^2(x_0) + 2 (b-a) \int_a^b \vert u'(s) \vert^2 \,d{s}.
\end{equation}
Then we integrate over $[a, b]$ with respect to $x$,
\begin{equation*}
	\int_a^b u^2(x) \,d{x} \leq 2 (b-a) u^2(x_0) + 2 (b-a)^2 \int_a^b \vert u'(s) \vert^2 \,d{s}.
\end{equation*}
Exchanging $x$ and $x_0$ in \eqref{eq:point-l2-estimation-tmp} then integrating over $[a, b]$ with respect to $x$, \red{it follows that,}
\begin{equation*}
	(b-a) u^2(x_0) \leq 2 \int_a^b u^2(x) \,d{x} + 2 (b-a)^2 \int_a^b \vert u'(x) \vert^2 \,d{x}.
\end{equation*}
\end{proof}

\red{When applying the above estimate to  specific subinterval, the following lemma is valid.}
\begin{lemma}\label{lemma:histogram-inf}
	For \red{nonnegative} integers $p, q$ with $0 \leq p < q \leq M$, let subinterval $I' := (pd, qd)$. Suppose \red{the indicator function on $I'$ can be bounded below by}
	\begin{equation*}
		\inf_{x \in I'} \rho_{N}(x) \geq \gamma_N > 0.
	\end{equation*}
Then, for arbitrary $u(x) \in C^1[0, 1]$, \blue{its $L^2$ norm can be estimated as}
	\begin{equation*}
		\Vert u \Vert_{L^2(I')}^2 \leq 2 \Big( \frac{1}{N \gamma_N} \sum_{i=1}^{N} u^2(x_i) + d^2 \Vert u' \Vert_{L^2(I')}^2 \Big).
	\end{equation*}
\end{lemma}

\begin{proof}
Note that $I' = (pd, qd)$ is the interior of $I_{p+1} \cup \cdots \cup I_{q}$.
For arbitrary $x \in I_j$ with $p+1\leq j\leq q$, \red{the application of Lemma \ref{lemma:point-l2-estimation} yields}
\begin{equation*}
	\Vert u \Vert_{L^2(I_j)}^2 \leq 2 \Big( d u^2(x) + d^2 \Vert u' \Vert_{L^2(I_j)}^2 \Big).
\end{equation*}
Substitute $x = x_i$ and \red{add up} all observation points that \red{belong to} $I_j$, it follows that,
\begin{equation}\label{eq:histogram-inf-tmp-1}
	\Vert u \Vert_{L^2(I_j)}^2 \leq 2 \Big( N_j^{-1} d \sum_{i=1}^{N} \mathbbm{1}_{x_i \in I_j} \cdot u^2(x_i) + d^2 \Vert u' \Vert_{L^2(I_j)}^2 \Big),
\end{equation}
where $\mathbbm{1}$ is characteristic function. \red{Referring to the definition} of indicator function $\rho_N(x)$ in \eqref{eq:rho-n-definition}, if it \red{can be bounded below by} $\gamma_N$, then $N_j (N d)^{-1} \geq \gamma_N$. Hence,
\begin{equation*}
	N_j^{-1} d \leq \frac{1}{N \gamma_N},\quad p+1 \leq j \leq q.
\end{equation*}
\red{Putting the above estimate} into \eqref{eq:histogram-inf-tmp-1}, it follows that,
\begin{equation*}
    \Vert u \Vert_{L^2(I_j)}^2 \leq 2 \Big( \frac{1}{N \gamma_N} \sum_{i=1}^{N} \mathbbm{1}_{x_i \in I_j} \cdot u^2(x_i) + d^2 \Vert u' \Vert_{L^2(I_j)}^2 \Big).
\end{equation*}
Finally, we sum up all the the above \red{estimates from $j=p+1$ to $j=q$}, \red{since the subintervals $I_j$ are disjoint}, we have
\begin{equation*}
	\Vert u \Vert_{L^2(I')}^2 \leq 2 \Big( \frac{1}{N \gamma_N} \sum_{i=1}^{N} u^2(x_i) + d^2 \Vert u' \Vert_{L^2(I')}^2 \Big).
\end{equation*}
\end{proof}

\begin{lemma}\label{lemma:histogram-sup}
	Suppose \red{the indicator function on $[0,1]$ has an upper bound}
	\begin{equation*}
		\sup_{x \in [0, 1]} \rho_N(x) \leq \beta_N.
	\end{equation*}
	Then, for $u(x) \in C^1[0, 1]$, \blue{its mean squared value at observation points can be estimated as}
	\begin{equation*}
		\frac{1}{N} \sum_{i=1}^{N} u^2(x_i) \leq 2 \beta_N \Big( \Vert u \Vert_{L^2(0, 1)}^2 + d^2 \Vert u' \Vert_{L^2(0, 1)}^2 \Big).
	\end{equation*}
\end{lemma}

\begin{proof}
Note that $[0, 1] = I_1 \cup \cdots \cup I_M$. For arbitrary $x \in I_j$ with $1\leq j\leq M$, \red{the application of Lemma \ref{lemma:point-l2-estimation} yields}
\begin{equation*}
	u^2(x) \leq \frac{2}{d} \Vert u \Vert_{L^2(I_j)}^2 + 2 d \Vert u' \Vert_{L^2(I_j)}^2.
\end{equation*}
\red{Substitute $x = x_i$ and add up} all observation points belong to $I_j$, \blue{it follows that}
\begin{equation}\label{eq:histogram-sup-tmp-1}
	\sum_{i=1}^{N} \mathbbm{1}_{x_i \in I_j} \cdot u^2(x_i) \leq \frac{2 N_j}{d} \Vert u \Vert_{L^2(I_j)}^2 + 2 N_j d \Vert u' \Vert_{L^2(I_j)}^2.
\end{equation}
Referring to the upper bound $\beta_N$ of the indicator function $\rho_N(x)$, it is obvious that $N_j (N d)^{-1} \leq \beta_N$ and consequently $N_j/d\leq N\beta_N$. Hence,
\begin{equation*}
	\sum_{i=1}^{N} \mathbbm{1}_{x_i \in I_j} \cdot u^2(x_i) \leq 2 N \beta_N \Vert u \Vert_{L^2(I_j)}^2 + 2 N \beta_N d^2 \Vert u' \Vert_{L^2(I_j)}^2.
\end{equation*}
Finally, we sum up all the the above \red{estimates from $j=1$ to $j=M$}, since the subintervals $I_j$ are \blue{disjoint}, we have
\begin{equation*}
	\frac{1}{N} \sum_{i=1}^{N} u^2(x_i) \leq 2 \beta_N \Vert u \Vert_{L^2(0, 1)}^2 + 2 \beta_N d^2 \Vert u' \Vert_{L^2(0, 1)}^2.
\end{equation*}
\end{proof}

\red{The last preliminary lemma provides} \blue{an estimate of} $L^2$ norm by function values at observation points and second order derivative. \red{We are willing to} replace $\Vert u' \Vert_{L^2(I')}^2$ on the right hand side of estimate in Lemma \ref{lemma:histogram-inf} by $\Vert u'' \Vert_{L^2(I')}^2$. \red{To this end,} the Sobolev inequality \eqref{eq:sobolev-interpolation} should be introduced and be utilized.

\begin{lemma}[Sobolev interpolation inequality {\cite[Theorem~5.2]{Adams.2003}}]
	Suppose that $I_s = (a, b)$.
	For arbitrary $u \in W^{2, 2}(I_s)$ and $\epsilon_0 > 0$, there exists \blue{a Sobolev constant $K = K(\epsilon_0, \vert I_s \vert)$}, such that for arbitrary $\epsilon \in (0, \epsilon_0]$, the $L^2$ norm of $u'$ can be estimated as
	\begin{equation}\label{eq:sobolev-interpolation}
		\Vert u' \Vert_{L^2(I_s)}^2 \leq K \Big( \epsilon^{-2} \Vert u \Vert_{L^2(I_s)}^2 + \epsilon^2 \Vert u'' \Vert_{L^2(I_s)}^2 \Big).
	\end{equation}
	\blue{In particular, if $\epsilon_0 = \vert I_s \vert = b-a$, the constant $K = K(\epsilon_0, \vert I_s \vert)$ can be replaced by}
	\begin{equation*}
		K_\ast = 32.
	\end{equation*}
\end{lemma}

\begin{lemma}\label{lemma:estimation-longer-interval}
For \red{nonnegative} integers $p, q$ with $0 \leq p < q \leq M$ \blue{and $q - p \geq 2 \sqrt{K_\ast}$}, let subinterval $I' := (pd, qd)$. 
Suppose the indicator function is bounded below on $I'$ by
\begin{equation*}
	\inf_{x \in I'} \rho_N(x) \geq \gamma_N > 0,
\end{equation*}
Then for $u(x) \in W^{2, 2}(0, 1)$, \blue{its $L^2$ norm on $I'$ can be estimated as}
	\begin{equation}\label{eq:estimation-longer-interval}
	\Vert u \Vert_{L^2(I')}^2 \leq \frac{4}{N \gamma_N} \sum_{i=1}^{N} u^2(x_i) + 16 K_\ast^2 d^4 \Vert u'' \Vert_{L^2(I')}^2.
	\end{equation}
\end{lemma}

\begin{proof}
\blue{We apply the Sobolev interpolation inequality \eqref{eq:sobolev-interpolation} with $I_s = I'$ and $\epsilon_0 = \vert I' \vert=(q-p)d$}. For $\epsilon \in (0, \epsilon_0]$, it follows that,
\begin{equation*}
\begin{split}
	\Vert u \Vert_{L^2(I')}^2 &\leq 2 \Big( \frac{1}{N \gamma_N} \sum_{i=1}^{N} u^2(x_i) + d^2 \Vert u' \Vert_{L^2(I')}^2 \Big) \\
	&\leq 2 \Big( \frac{1}{N \gamma_N} \sum_{i=1}^{N} u^2(x_i) + d^2 K_\ast \big( \epsilon^{-2} \Vert u \Vert_{L^2(I')}^2 + \epsilon^2 \Vert u'' \Vert_{L^2(I')}^2 \big) \Big) \\
	&= \frac{2}{N \gamma_N} \sum_{i=1}^{N} u^2(x_i) + 2 K_\ast \Big( \frac{d}{\epsilon} \Big)^2 \Vert u \Vert_{L^2(I')}^2 + 2 K_\ast d^2 \epsilon^2 \Vert u'' \Vert_{L^2(I')}^2.
\end{split}
\end{equation*}
For particular
\begin{equation*}
	\epsilon = 2 K_\ast^{1/2} d \leq \epsilon_0,
\end{equation*}
the coefficient $2 K_\ast (d/\epsilon)^2$ \red{on the right hand side} satisfies
\begin{equation*}
	2 K_\ast \Big( \frac{d}{\epsilon} \Big)^2 \leq \frac{1}{2}.
\end{equation*}
Therefore,
\begin{equation*}
	\Vert u \Vert_{L^2(I')}^2 \leq \frac{4}{N \gamma_N} \sum_{i=1}^{N} u^2(x_i) + 16 K_\ast^2 d^4 \Vert u'' \Vert_{L^2(I')}^2.
\end{equation*}
\end{proof}

Let
\begin{align*}
e_N(x) = f_N(x)-f(x)
\end{align*}
be the error function of the proposed regularization algorithm, \blue{and
\begin{equation*}
    \bm{\eta}_N = (\eta_1, \eta_2, \cdots, \eta_N)^\top \in \mathbb{R}^{N}
\end{equation*}
be the vector of random noise}. \red{The error analysis can be discussed separately by introducing the deterministic part
\begin{align*}
		&f_{N, 1} = \argmin_{g \in V_M} J(g; \alpha_N, \bm{y}_N - \bm{\eta}_N) = \Phi[\bm{\lambda}_{N, 1}],\\
        &e_{N, 1} = f_{N, 1} - f,
	\end{align*}
and the random part
\begin{equation*}
		f_{N, 2} = \argmin_{g \in V_M} J(g; \alpha_N, \bm{\eta}_N) = \Phi[\bm{\lambda}_{N, 2}].
\end{equation*}}
Thanks to the linearity of Tikhonov regularization, we have
$$f_N = f_{N, 1} + f_{N, 2},\quad\textrm{and} \ e_N = e_{N, 1} + f_{N, 2}.$$

Referring to Lemma \ref{lemma:estimation-longer-interval}, in order to discuss the $L^2$ norms of $e_{N, 1}$ and $f_{N, 2}$ respectively, we will estimate their mean squared errors at observation points and $L^2$ norms of second order derivatives. It is necessary to provide the following lemma.
\begin{lemma}[{\cite[Theorem~1.55]{Micula.1999}}]\label{lemma:cubic-spline-interpolation-error}
	Suppose that $f(x) \in W^{2, 2}(0, 1)$, let $s_{f, M} \in V_M$ be the natural cubic spline interpolant of $f$ with knots $\{p_j\}_{j=0}^{M}$.
	Then, the $L^2$ norms of $f''$ and $s''_{f, M}$ satisfy the following \blue{equality}
	\begin{equation}\label{eq:cubic-spline-pythagorean-equality}
		\Vert s''_{f, M} \Vert_{L^2(0, 1)}^2 + \Vert s''_{f, M} - f'' \Vert_{L^2(0, 1)}^2 = \Vert f'' \Vert_{L^2(0, 1)}^2.
	\end{equation}
	The interpolation errors of $s_{f, M} - f$ and $s'_{f, M} - f'$ \blue{can be estimated as}
	\begin{equation}\label{eq:cubic-spline-error-estimate}
		\begin{split}
		\Vert s_{f, M} - f \Vert_{L^2(0, 1)}^2 &\leq \frac{d^4}{16} \Vert s''_{f, M} - f'' \Vert_{L^2(0, 1)}^2 \leq \frac{d^4}{16} \Vert f'' \Vert_{L^2(0, 1)}^2, \\
		\Vert s'_{f, M} - f' \Vert_{L^2(0, 1)}^2 &\leq \frac{d^2}{2} \Vert s''_{f, M} - f'' \Vert_{L^2(0, 1)}^2 \leq \frac{d^2}{2} \Vert f'' \Vert_{L^2(0, 1)}^2.
		\end{split}
	\end{equation}
\end{lemma}

\begin{lemma}
	Suppose that $f(x) \in W^{2, 2}(0, 1)$, and \blue{the indicator function is bounded above by}
	\begin{equation*}
		\sup_{x \in [0, 1]} \rho_N(x) \leq \beta_N.
	\end{equation*}
	\blue{Then, the mean squared value of $e_{N, 1}$ and the $L^2$ norm of $e''_{N, 1}$ can be estimated as}
	\begin{equation*}
		\begin{split}
		\frac{1}{N} \sum_{i=1}^{N} e_{N, 1}^2(x_i) &\leq \frac{9}{8} \beta_N d^4 \Vert f'' \Vert_{L^2(0, 1)}^2 + \alpha_N \Vert f'' \Vert_{L^2(0, 1)}^2, \\
		\Vert e_{N, 1}'' \Vert_{L^2(0, 1)}^2 &\leq \frac{9}{4} \beta_N \Vert f'' \Vert_{L^2(0, 1)}^2 \cdot \frac{d^4}{\alpha_N} + 4 \Vert f'' \Vert_{L^2(0, 1)}^2.
		\end{split}
	\end{equation*}
\end{lemma}

\begin{proof}
	Denote that
	\begin{equation*}
		E_N = \sum_{i=1}^{N} \big( s_{f, M}(x_i) - f(x_i) \big)^2.
	\end{equation*}
	Referring to Lemma \ref{lemma:histogram-sup} and \eqref{eq:cubic-spline-error-estimate}, we have
	\begin{equation*}
		\begin{split}
		E_N &\leq 2 \beta_N \Big( \Vert f - s_{f, M} \Vert_{L^2(0, 1)}^2 + d^2 \Vert f' - f'_M \Vert_{L^2(0, 1)}^2 \Big) \\
		&\leq 2 \beta_N \Big( \frac{d^4}{16} \Vert f'' \Vert_{L^2(0, 1)}^2 + \frac{d^4}{2} \Vert f'' \Vert_{L^2(0, 1)}^2 \Big) \\
		&\leq \frac{9}{8} \beta_N d^4 \Vert f'' \Vert_{L^2(0, 1)}^2.
		\end{split}
	\end{equation*}

	Recalling that $f_{N, 1}$ minimizes the functional $J(\cdot; \alpha_N, \bm{y}_N - \bm{\eta}_N)$ in $V_M$, therefore,
	\begin{equation}\label{eq:truth-part-minimality-of-tikhonov}
		\frac{1}{N} \sum_{i=1}^{N} e_{N, 1}^2(x_i) + \alpha_N \Vert f_{N, 1}'' \Vert_{L^2(0, 1)}^2 \leq E_N + \alpha_N \Vert s_{f, M}'' \Vert_{L^2(0, 1)}^2.
	\end{equation}
	Combining the conclusion in \eqref{eq:cubic-spline-pythagorean-equality} yields that
	\begin{equation*}
		\frac{1}{N} \sum_{i=1}^{N} e_{N, 1}^2(x_i) \leq E_N + \alpha_N \Vert f'' \Vert_{L^2(0, 1)}^2 \leq \frac{9}{8} \beta_N d^4 \Vert f'' \Vert_{L^2(0, 1)}^2 + \alpha_N \Vert f'' \Vert_{L^2(0, 1)}^2.
	\end{equation*}
	On the other hand, \blue{divide both sides of \eqref{eq:truth-part-minimality-of-tikhonov} by $\alpha_N$, it follows that}
	\begin{equation*}
		\Vert f_{N, 1}'' \Vert_{L^2(0, 1)}^2 \leq \frac{E_N}{\alpha_N} + \Vert s_{f, M}'' \Vert_{L^2(0, 1)}^2.
	\end{equation*}
	Therefore,
	\begin{equation*}
		\begin{split}
		\Vert e_{N, 1}'' \Vert_{L^2(0, 1)}^2 &\leq 2 \Vert f_{N, 1}'' \Vert_{L^2(0, 1)}^2 + 2 \Vert f'' \Vert_{L^2(0, 1)}^2 \\
		&\leq 2 \frac{E_N}{\alpha_N} + 2 \Vert s_{f, M}'' \Vert_{L^2(0, 1)}^2 + 2 \Vert f'' \Vert_{L^2(0, 1)}^2 \\
		&\leq \frac{9}{4} \beta_N \Vert f'' \Vert_{L^2(0, 1)}^2 \cdot \frac{d^4}{\alpha_N} + 4 \Vert f'' \Vert_{L^2(0, 1)}^2.
		\end{split}
	\end{equation*}
\end{proof}

\begin{corollary}
	Suppose that $f(x) \in W^{2, 2}(0, 1)$, and \blue{the indicator function is bounded above by}
	\begin{equation*}
		\sup_{x \in [0, 1]} \rho_N(x) \leq \beta_N.
	\end{equation*}
	\red{Choosing the regularization parameter}
	\begin{equation*}
		\alpha_N = \frac{M \sigma^2}{N} + d^4,
	\end{equation*}
	\blue{the mean squared value of $e_{N, 1}$ and the $L^2$ norm of $e''_{N, 1}$ can be estimated as}
	\begin{equation}\label{eq:truth-part-mse-estimates}
		\begin{split}
		\frac{1}{N} \sum_{i=1}^{N} e_{N, 1}^2(x_i) &\leq \frac{M \sigma^2}{N} \Vert f'' \Vert_{L^2(0, 1)}^2 + d^4 \Big( \frac{9}{8} \beta_N + 1 \Big) \Vert f'' \Vert_{L^2(0, 1)}^2, \\
		\Vert e_{N, 1}'' \Vert_{L^2(0, 1)}^2 &\leq \Big( \frac{9}{4} \beta_N + 4 \Big) \Vert f'' \Vert_{L^2(0, 1)}^2.
		\end{split}
	\end{equation}
\end{corollary}

% we recall and introduce some useful notations.
%In Definition \ref{definition:determistic-random-parts}, we defined the random part as
%\begin{equation*}
%    f_{N, 2} = \argmin_{g \in V_M} J(g; \alpha_N, \bm{\eta}_N) = \Phi[\bm{\lambda}_{N, 2}],
%\end{equation*}
%where $\bm{\lambda}_{N, 2} \in \mathbb{R}^{M+3}$.
%By Theorem \ref{thm:unique-minimizer-of-tikhonov-functional-in-vm}, $\bm{\lambda}_{N, 2}$ is given by
%\begin{equation*}
%    \Big( \alpha_N P + \frac{1}{N} H_N^\top H_N \Big) \bm{\lambda}_{N, 2} = \frac{1}{N} H_N^\top \bm{\eta}_N,
%\end{equation*}
%that is,
%\begin{equation*}
%    \bm{\lambda}_{N, 2} = \big( N \alpha_N P + H_N^\top H_N \big)^{-1} H_N^\top \bm{\eta}_N.
%\end{equation*}
\blue{For the random part $f_{N, 2} = \Phi[\bm{\lambda}_{N, 2}]$, recalling Theorem \ref{thm:unique-minimizer-of-tikhonov-functional-in-vm}, $\bm{\lambda}_{N, 2}$ can be written as
\begin{equation}
    \bm{\lambda}_{N, 2} = \big( N \alpha_N P + H_N^\top H_N \big)^{-1} H_N^\top \bm{\eta}_N,
\end{equation}
where $\bm{\eta}_N \in \mathbb{R}^{N}$ is the vector of random noise.
}
What we need to estimate are the mean squared value of $f_{N, 2}$ at observation points
\begin{equation*}
    \frac{1}{N} \sum_{i=1}^{N} f_{N, 2}^2(x_i) = \frac{1}{N}  \big\Vert H_N \bm{\lambda}_{N, 2} \big\Vert_2^2,
\end{equation*}
and the $L^2$ norm of $f_{N, 2}$ on interval $(0, 1)$
\begin{equation*}
    \Vert f''_{N, 2} \Vert_{L^2(0, 1)}^2 = \bm{\lambda}_{N, 2}^\top P \bm{\lambda}_{N, 2}.
\end{equation*}

\red{A difficulty is} the matrix $P$ is positive semidefinite and not invertible.
\red{In order to solve} this problem, we disturb $P$ by identity matrix $\mathbb{I}$.
\blue{Let}
\begin{equation}\label{eq:definition-of-p-epsilon-and-lambda-epsilon-1d}
\begin{split}
    P_\epsilon &= P + \epsilon \mathbb{I}, \\
    \bm{\lambda}^{\epsilon}_{N, 2} &= ( N \alpha_N P_\epsilon +  H_N^\top H_N )^{-1} H_N^\top \bm{\eta}_N,
\end{split}
\end{equation}
\red{in which $\epsilon$ is a small nonnegative constant}, when $\epsilon = 0$, we have $P_\epsilon = P$, $\bm{\lambda}_{N, 2}^{\epsilon} = \bm{\lambda}_{N, 2}$. \red{Our idea is to }prove the desired results \red{provided that} $\epsilon > 0$, then let $\epsilon \to 0^+$.

\red{Before the formal analysis, the following lemmas are necessary.}

\begin{lemma}[Woodbury matrix identity]
	Let $A \in \mathbb{R}^{n \times n}$ be an invertible matrix, $U \in \mathbb{R}^{n \times k}$, $C \in \mathbb{R}^{k \times k}$, $V \in \mathbb{R}^{k \times n}$.
	Then, there holds
	\begin{equation}\label{eq:woodbury-matrix-identity}
		(A + U C V)^{-1} = A^{-1} - A^{-1} U (C^{-1} + V A^{-1} U)^{-1} V A^{-1}.
	\end{equation}
\end{lemma}

\begin{lemma}[Fatou's lemma]\label{lemma:fatou-lemma}
	Let $X_1, X_2, \cdots$ be a sequence of nonnegative random variables. Then, there holds
	\begin{equation*}
		\mathbb{E}\Big[ \liminf_{n \to \infty} X_n \Big] \leq \liminf_{n \to \infty} \mathbb{E}[X_n].
	\end{equation*}
\end{lemma}

\begin{lemma}[Markov's inequality]
	Suppose that $X$ is a nonnegative random variable, and $a > 0$.
	Then, \red{the following inequality is satisfied,}
	\begin{equation}\label{eq:markov-inequality}
		\mathbb{P} (X \geq a) \leq \frac{\mathbb{E}[X]}{a}.
	\end{equation}
\end{lemma}

\red{Based on the above preparation, we introduce the following lemma, which estimate the mean squared value of the random part $f_{N, 2}$ and the $L^2$ norm of $f''_{N, 2}$ on interval $(0, 1)$} \blue{with $P$ replaced by $P_\epsilon$}.
\begin{lemma}
	For $\epsilon > 0$, \blue{there holds}
	\begin{equation}\label{eq:noise-part-perturbed-estimates}
		\begin{split}
		\mathbb{E} \Vert H_N \bm{\lambda}_{N, 2}^{\epsilon} \Vert_2^2 &\leq \sigma^2 (M+3), \\
		\mathbb{E} \Big[  (\bm{\lambda}_{N, 2}^\epsilon)^\top P_\epsilon \bm{\lambda}_{N, 2}^\epsilon \Big] &\leq \frac{\sigma^2 (M+3)}{4 N \alpha_N}.
		\end{split}
	\end{equation}
\end{lemma}

\begin{proof}
(P1) Let
\begin{equation*}
	S = H_N P_\epsilon^{-1} H_N^\top,
\end{equation*}
then $S$ is a positive semidefinite matrix.
Denote the eigen-decomposition of $S$ by
\begin{equation*}
	S = U T U^\top,
\end{equation*}
\red{in which} $U \in \mathbb{R}^{N \times N}$ is an orthogonal matrix, $T \in \mathbb{R}^{N \times N}$ is a diagonal matrix that is composed of all eigenvalues of $S$ in a nonascending order, i.e.,
\begin{equation*}
	T = \mathrm{diag}\{t_1, t_2, \cdots, t_M, t_{M+1}, t_{M+2}, t_{M+3}, \cdots, t_N\},
\end{equation*}
where $t_i \geq t_{i+1}$, $1 \leq i \leq N-1$.
Since $\mathrm{rank}(S) \leq \mathrm{rank}(P_\epsilon^{-1}) \leq M+3$, \blue{the number of nonzero eigenvalues cannot exceed $M+3$, hence}
\begin{equation*}
	t_{i} = 0,\quad M+4 \leq i \leq N.
\end{equation*}

(P2) Since $\eta_i, 1 \leq i \leq N$ are uncorrelated random variables,
\begin{equation*}
	\mathbb{E}[\bm{\eta}_N \bm{\eta}_N^\top] = \sigma^2 \mathbb{I}_N
\end{equation*}
\blue{is satisfied}.

(P3) We need \blue{an equivalent} expression of $\bm{\lambda}_{N, 2}^{\epsilon}$.
\red{Referring to} the definition of $\bm{\lambda}_{N, 2}^{\epsilon}$ \eqref{eq:definition-of-p-epsilon-and-lambda-epsilon-1d} and Woodbury matrix identity \eqref{eq:woodbury-matrix-identity}, we have
\begin{equation*}
    \begin{split}
    \bm{\lambda}^\epsilon_{N, 2} &= \big( N \alpha_N P_\epsilon + H_N^\top H_N \big)^{-1} H_N^\top \bm{\eta}_N \\
        &= (N \alpha_N P_\epsilon)^{-1} H_N^\top \bm{\eta}_N \\
        &\qquad - (N \alpha_N P_\epsilon)^{-1} H_N^\top \big( \mathbb{I} + H_N (N \alpha_N P_\epsilon)^{-1} H_N^\top \big)^{-1} H_N (N \alpha_N P_\epsilon)^{-1} H_N^\top \bm{\eta}_N \\
        &= (N \alpha_N)^{-1} P_\epsilon^{-1} H_N^\top \Big[ \mathbb{I} - \big( (N \alpha_N) \mathbb{I} + H_N P_\epsilon^{-1} H_N^\top \big)^{-1} H_N P_\epsilon^{-1} H_N^\top \Big] \bm{\eta}_N \\
        &= (N \alpha_N)^{-1} P_\epsilon^{-1} H_N^\top \big( (N \alpha_N) \mathbb{I} + H_N P_\epsilon^{-1} H_N^\top \big)^{-1} \Big[ (N \alpha_N) \mathbb{I} + H_N P_\epsilon^{-1} H_N^\top - H_N P_\epsilon^{-1} H_N^\top \Big] \bm{\eta}_N \\
        &= P_\epsilon^{-1} H_N^\top \big( (N \alpha_N) \mathbb{I} + H_N P_\epsilon^{-1} H_N^\top \big)^{-1} \bm{\eta}_N \\
        &= P_\epsilon^{-1} H_N^\top \big( (N \alpha_N) \mathbb{I} + S \big)^{-1}\bm{\eta}_N.
    \end{split}
\end{equation*}

By (P1), (P2) and (P3), we have
\begin{equation*}
	\begin{split}
	\mathbb{E} \Vert H_N \bm{\lambda}_{N, 2}^{\epsilon} \Vert_2^2 &= \mathbb{E} \Big[ (\bm{\lambda}_{N, 2}^{\epsilon})^\top H_N^\top H_N \bm{\lambda}_{N, 2}^{\epsilon} \Big] \\
	&= \mathbb{E} \Big[ \bm{\eta}_N^\top \big( (N \alpha_N) \mathbb{I} + S \big)^{-1} S^2 \big( (N \alpha_N) \mathbb{I} + S \big)^{-1} \bm{\eta}_N \Big] \\
	&= \mathbb{E} \Big[ \mathrm{tr} \Big( S^2 \big( (N \alpha_N) \mathbb{I} + S \big)^{-2} \bm{\eta}_N \bm{\eta}_N^\top \Big) \Big] \\
	&= \mathrm{tr} \Big( S^2 \big( (N \alpha_N) \mathbb{I} + S \big)^{-2} \mathbb{E} [\bm{\eta}_N \bm{\eta}_N^\top] \Big) \\
	&= \sigma^2 \mathrm{tr} \Big( S^2 \big( (N \alpha_N) \mathbb{I} + S \big)^{-2} \Big) \\
	&= \sigma^2 \sum_{i=1}^{N} \frac{t_i^2}{(t_i + N \alpha_N)^2} \leq \sigma^2 (M+3).
	\end{split}
\end{equation*}
\red{The second estimate is derived by}
\begin{equation*}
	\begin{split}
	\mathbb{E} \Big[ (\bm{\lambda}_{N, 2}^{\epsilon})^\top P_\epsilon \bm{\lambda}_{N, 2}^{\epsilon} \Big] &= \mathbb{E} \Big[ \bm{\eta}_N^\top \big( (N \alpha_N) \mathbb{I} + S \big)^{-1} H_N P_\epsilon^{-1} P_\epsilon P_\epsilon^{-1} H_N^\top \big( (N \alpha_N) \mathbb{I} + S \big)^{-1} \bm{\eta}_N \Big] \\
	&= \mathbb{E} \Big[ \mathrm{tr} \Big( S \big( (N \alpha_N) \mathbb{I} + S \big)^{-2} \bm{\eta}_N \bm{\eta}_N^\top \Big) \Big] \\
	&= \sigma^2 \mathrm{tr} \Big( S \big( (N \alpha_N) \mathbb{I} + S \big)^{-2} \Big) \\
	&= \sigma^2 \sum_{i=1}^{N} \frac{t_i}{(t_i + N \alpha_N)^2} \\
	&\leq \sigma^2 \sum_{i=1}^{N} \frac{t_i}{4 t_i N \alpha_N} \leq \frac{\sigma^2 (M+3)}{4 N \alpha_N}.
	\end{split}
\end{equation*}
\end{proof}

\red{Then, let $\epsilon \to 0^+$, the application of  Fatou's lemma gives the following lemma.  }

\begin{lemma}\label{lemma:mse-noise-part}
	\blue{The mean squared value of $f_{N, 2}$ and the $L^2$ norm of $f''_{N, 2}$ can be estimated as}
	% For the random part $f_{N, 2}$, there holds
	\begin{equation*}
		\begin{split}
		\mathbb{E} \Big[ \frac{1}{N} \sum_{i=1}^{N} f_{N, 2}^2(x_i) \Big] &\leq \frac{\sigma^2(M+3)}{N}, \\
		\mathbb{E} \Vert f_{N, 2}'' \Vert_{L^2(0, 1)}^2 &\leq \frac{\sigma^2 (M+3)}{4 N \alpha_N}.
		\end{split}
	\end{equation*}
\end{lemma}

\begin{proof}
We first show that $\bm{\lambda}_{N, 2} = \lim_{\epsilon \to 0^+} \bm{\lambda}_{N, 2}^{\epsilon}$ holds almost surely.
Denote
\begin{equation*}
	B = N \alpha_N P + H_N^\top H_N,
\end{equation*}
\red{referring to } Theorem \ref{thm:unique-minimizer-of-tikhonov-functional-in-vm}, \red{the matrix} $B$ is positive definite. 
% Thus,
% \begin{equation*}
% 	\lambda_{\max}(B) \geq \lambda_{\min}(B) > 0,
% \end{equation*}
% where $\lambda_{\max}(B)$ and $\lambda_{\min}(B)$ \red{contains} the largest and the smallest eigenvalue of $B$, respectively.
\blue{Thus, denote by $\lambda_{\min}(B)$ the smallest eigenvalue of $B$, we have $\lambda_{\min}(B) > 0$.}
Since
\begin{equation*}
	\begin{split}
	\bm{\lambda}_{N, 2} &= B^{-1} H_N^\top \bm{\eta}_N, \\
	\bm{\lambda}_{N, 2}^{\epsilon} &= ( N \alpha_N P_\epsilon + H_N^\top H_N )^{-1} H_N^\top \bm{\eta}_N = (B + \epsilon N \alpha_N \mathbb{I})^{-1} H_N^\top \bm{\eta}_N,
	\end{split}
\end{equation*}
we have
\begin{equation*}
	(B + \epsilon N \alpha_N \mathbb{I}) (\bm{\lambda}_{N, 2}^\epsilon - \bm{\lambda}_{N, 2}) = - \epsilon N \alpha_N \bm{\lambda}_{N, 2}.
\end{equation*}
\blue{It follows that}
\begin{equation*}
	\Vert \bm{\lambda}_{N, 2}^\epsilon - \bm{\lambda}_{N, 2} \Vert_2^2 \leq \epsilon^2 (N \alpha_N)^2 \Vert (B + \epsilon N \alpha_N \mathbb{I})^{-1} \Vert_2^2 \Vert \bm{\lambda}_{N, 2} \Vert_2^2 \leq \epsilon^2 \frac{(N \alpha_N)^2}{\lambda_{\min}^2} \Vert \bm{\lambda}_{N, 2} \Vert_2^2.
\end{equation*}
Hence, 
\begin{equation*}
	\lim_{\epsilon \to 0^+} \bm{\lambda}_{N, 2}^\epsilon = \bm{\lambda}_{N, 2},\quad\text{a.s.}
\end{equation*}

\red{Therefore}, the following two equalities
\begin{equation*}
	\begin{split}
	\frac{1}{N} \sum_{i=1}^{N} f_{N, 2}^2(x_i) &= \frac{1}{N} \Vert H_N \bm{\lambda}_{N, 2} \Vert_2^2 = \frac{1}{N} \lim_{\epsilon \to 0^+} \Vert H_N \bm{\lambda}_{N, 2}^\epsilon \Vert_2^2, \\
	\lim_{\epsilon \to 0^+} (\bm{\lambda}_{N, 2}^{\epsilon})^\top P_\epsilon \bm{\lambda}_{N, 2}^{\epsilon} &= \lim_{\epsilon \to 0^+} (\bm{\lambda}_{N, 2}^{\epsilon})^\top P \bm{\lambda}_{N, 2}^{\epsilon} + \lim_{\epsilon \to 0^+} \epsilon \Vert \bm{\lambda}_{N, 2}^{\epsilon} \Vert_2^2 = \bm{\lambda}_{N, 2}^\top P \bm{\lambda}_{N, 2}
	\end{split}
\end{equation*}
almost surely hold. \red{In addition, the application of} Fatou's lemma yields
\begin{equation*}
\begin{split}
	\mathbb{E} \Big[ \frac{1}{N} \sum_{i=1}^{N} f_{N, 2}^2(x_i) \Big] &= \frac{1}{N} \mathbb{E} \Big[ \lim_{\epsilon \to 0^+} \Vert H_N \bm{\lambda}_{N, 2}^\epsilon \Vert_2^2 \Big] \\
	&\leq \frac{1}{N} \liminf_{\epsilon \to 0^+} \mathbb{E} \Vert H_N \bm{\lambda}_{N, 2}^{\epsilon} \Vert_2^2 \leq \frac{\sigma^2 (M+3)}{N},
\end{split}
\end{equation*}
and
\begin{equation*}
	\begin{split}
	\mathbb{E} \Vert f_{N, 2}'' \Vert_{L^2(0, 1)}^2 &= \mathbb{E} \Big[ \bm{\lambda}_{N, 2}^\top P \bm{\lambda}_{N, 2} \Big] \\
	&= \mathbb{E} \Big[ \lim_{\epsilon \to 0^+} (\bm{\lambda}_{N, 2}^{\epsilon})^\top P_\epsilon \bm{\lambda}_{N, 2}^{\epsilon} \Big] \\
	&\leq \liminf_{\epsilon \to 0^+} \mathbb{E} \Big[  (\bm{\lambda}_{N, 2}^\epsilon)^\top P_\epsilon \bm{\lambda}_{N, 2}^\epsilon \Big] \\
	&\leq \frac{\sigma^2 (M+3)}{4 N \alpha_N}.
	\end{split}
\end{equation*}
\end{proof}

Finally, by Markov's inequality, the following confidence interval estimates are valid.

\begin{corollary}\label{corollary:mse-noise-part-confidence-interval}
	Suppose that $M \geq 3$.
	\red{Choosing the regularization parameter} $\alpha_N  = M \sigma^2 / N + d^4$, for arbitrary $\delta \in (0, 1)$, \blue{the estimates for the mean squared value of $f_{N, 2}$ and the $L^2$ norm of $f''_{N, 2}$}
	\begin{equation}\label{eq:noise-part-mse-estimates}
		\begin{split}
		\frac{1}{N} \sum_{i=1}^{N} f_{N, 2}^2(x_i) &\leq \frac{4 \sigma^2 M}{\delta N}, \\
		\Vert f_{N, 2}'' \Vert_{L^2(0, 1)}^2 &\leq \frac{1}{\delta}
		\end{split}
	\end{equation}
	hold with a probability of at least $1 - \delta$.
\end{corollary}

\begin{proof}
	We substitute $M \geq 3$ and the prior choice rule of $\alpha_N$ in the previous lemma, then a direct application of Markov's inequality \eqref{eq:markov-inequality} gives the results.
\end{proof}
% subsubsection random_part (end)

% subsection mean_squared_error_at_observation_points (end)

\subsection{Error bounds in continuous \texorpdfstring{$L^2$}{L2} norms} % (fold)
\label{sub:error_bounds_in_continuous_l2_norms}

 We substitute the above estimates into \blue{the estimate in} Lemma \ref{lemma:estimation-longer-interval} to obtain error bounds in continuous $L^2$ norms.

\begin{theorem}\label{theorem:error-bounds-given-n}
	Suppose that $f \in W^{2, 2}(0, 1)$, $M \geq 3$, $N \geq M$, \blue{the indicator function is bounded above by}
	\begin{equation*}
		\sup_{x \in [0, 1]} \rho_N(x) \leq \beta_N.
	\end{equation*}
	\blue{For nonnegative integers p, q with $0 \leq p < q \leq M$ and $q - p \geq 2 \sqrt{K_\ast}$, let subinterval $I' := (pd, qd)$.}
	\blue{Choosing the regularization parameter $\alpha_N = M \sigma^2 / N + d^4$, if the indicator function is bounded below on $I'$ by}
	\begin{equation*}
        \inf_{x \in I'} \rho_{N}(x) \geq \gamma_N > 0,
    \end{equation*}
	then for arbitrary $\delta \in (0, 1)$, \blue{the following estimates for the $L^2$ norms of $e_N$ and $e'_N$}
	\begin{equation*}
		\begin{split}
		\Vert e_N \Vert_{L^2(I')} &\leq C_1 \Big( \frac{M \sigma^2}{N} \Big)^{\frac{1}{2}} + C_2 M^{-2}, \\
		\Vert e'_N \Vert_{L^2(I')} &\leq C_3 \Big( \frac{M \sigma^2}{N} \Big)^{\frac{1}{4}} + C_4 M^{-1}
		\end{split}
	\end{equation*}
	hold with a probability of at least $1 - \delta$, where the constants $C_1, C_2, C_3$ and $C_4$ are independent of $M$ and $N$.
	\blue{In Sobolev interpolation inequality \eqref{eq:sobolev-interpolation} with $I_s = I'$ and $\epsilon_0 = \max \{\sqrt{\sigma}, (q-p) d\}$, denoting by $K_\sigma$ the Sobolev constant $K(\epsilon_0, \vert I' \vert)$, the constants can be written as}
	% Denote by $K_\sigma$ the Sobolev constant $K(\epsilon_0, (q-p) d)$ in Sobolev interpolation inequality \eqref{eq:sobolev-interpolation} where
	% \begin{equation*}
	% 	\epsilon_0 = \max \Big\{ \sigma^{\frac{1}{2}}, (q-p) d \Big\},
	% \end{equation*}
	% the constants read
	\begin{equation*}
		\begin{split}
		C_1 &= 2 \gamma_N^{-\frac{1}{2}} \Vert f'' \Vert_{L^2(0, 1)} + \frac{4}{\sqrt{\delta \gamma_N}}, \\
		C_2 &= \Vert f'' \Vert_{L^2(0, 1)} \sqrt{\frac{9 \beta_N + 8}{2 \gamma_N} + 36 K_\ast^2 \beta_N + 64 K_\ast^2} + 4 K_\ast \delta^{-\frac{1}{2}}, \\
		C_3 &= K_\sigma^{\frac{1}{2}} \Big( \Vert f'' \Vert_{L^2(0, 1)} \sqrt{4 \gamma_N^{-1} + \frac{9}{4} \beta_N + 4} + \delta^{-\frac{1}{2}} \sqrt{16 \gamma_N^{-1} + 1} \Big) , \\
		C_4 &= K_\sigma^{\frac{1}{2}} \Big( \Vert f'' \Vert_{L^2(0, 1)} \sqrt{\frac{9 \beta_N + 8}{8 K_\ast \gamma_N} + 18 K_\ast \beta_N + 32 K_\ast} + 2 \sqrt{2} \delta^{-\frac{1}{2}} K_\ast^{\frac{1}{2}} \Big).
		\end{split}
	\end{equation*}
\end{theorem}

\begin{proof}
	First, we derive the error bounds for function reconstruction.
	For deterministic part $e_{N, 1} = f_{N, 1} - f$, we substitute \eqref{eq:truth-part-mse-estimates} into \eqref{eq:estimation-longer-interval}, \blue{it follows that}
	\begin{equation*}
		\begin{split}
		\Vert e_{N, 1} \Vert_{L^2(I')}^2 &\leq \frac{4}{N \gamma_N} \sum_{i=1}^{N} e_{N, 1}^2(x_i) + 16 K_\ast^2 d^4 \Vert e_{N, 1}'' \Vert_{L^2(I')}^2 \\
		&\leq \frac{4}{\gamma_N} \cdot \Big[ \frac{M \sigma^2}{N} + d^4 \Big( \frac{9}{8} \beta_N + 1 \Big) \Big] \Vert f'' \Vert_{L^2(0, 1)}^2 + 16 K_\ast^2 d^4 \cdot \Big( \frac{9}{4} \beta_N + 4 \Big) \Vert f'' \Vert_{L^2(0, 1)}^2 \\
		&\leq \frac{M \sigma^2}{N} \cdot \frac{4}{\gamma_N} \Vert f'' \Vert_{L^2(0, 1)}^2 + d^4 \cdot \Big( \frac{9 \beta_N + 8}{2 \gamma_N} + 36 K_\ast^2 \beta_N + 64 K_\ast^2 \Big) \Vert f'' \Vert_{L^2(0, 1)}^2.
		\end{split}
	\end{equation*}
	For random part $f_{N, 2}$, first by Corollary \ref{corollary:mse-noise-part-confidence-interval}, \eqref{eq:noise-part-mse-estimates} holds with a probability of at least $1 - \delta$.
	Then we substitute \eqref{eq:noise-part-mse-estimates} into into \eqref{eq:estimation-longer-interval}, \blue{it follows that}
	\begin{equation*}
		\begin{split}
		\Vert f_{N, 2} \Vert_{L^2(I')}^2 &\leq \frac{4}{N \gamma_N} \sum_{i=1}^{N} f_{N, 2}^2(x_i) + 16 K_\ast^2 d^4 \Vert f''_{N, 2} \Vert_{L^2(I')}^2 \\
		&\leq \frac{M \sigma^2}{N} \cdot \frac{16}{\delta \gamma_N} + d^4 \cdot \frac{16 K_\ast^2}{\delta}.
		\end{split}
	\end{equation*}
	By triangular inequality, we have
	\begin{equation*}
		\Vert e_N \Vert_{L^2(I')} \leq \Vert e_{N, 1} \Vert_{L^2(I')} + \Vert f_{N, 2} \Vert_{L^2(I')} \leq C_1 \Big( \frac{M \sigma^2}{N} \Big)^{\frac{1}{2}} + C_2 M^{-2}.
	\end{equation*}

	Next, we derive error estimates for derivative reconstruction.
	\blue{To give estimates of the $L^2$ norms of $e'_{N, 1}$ and $f'_{N, 2}$, we apply Sobolev interpolation inequality \eqref{eq:sobolev-interpolation} with $I_s = I'$ and $\epsilon_0 = \max \{\sqrt{\sigma}, (q-p) d\}$, and choose}
	% We apply Sobolev interpolation inequality \eqref{eq:sobolev-interpolation} with
	\begin{equation*}
		\epsilon = \max\left\{ \Big( \frac{M \sigma^2}{N} \Big)^{\frac{1}{4}}, (q-p) d \right\}.
	\end{equation*}
	\blue{Since $N \geq M$ and $q-p \geq 2 \sqrt{K_\ast}$, the following estimates hold for $\epsilon$:}
	% Also, note that $(q-p) d \geq 2 K_\ast^{1/2} d$, we have the following estimates for $\epsilon$
	\begin{equation*}
		\epsilon \leq \epsilon_0,\quad \epsilon^2 \leq \Big( \frac{M \sigma^2}{N} \Big)^{\frac{1}{2}} + 4 K_\ast d^2,\quad \epsilon^{-2} \leq \min \left\{ \Big( \frac{M \sigma^2}{N} \Big)^{-\frac{1}{2}}, ( 4 K_\ast d^2 )^{-1} \right\}.
	\end{equation*}
	\blue{Therefore, we can give the following estimate of the $L^2$ norm of $e'_{N, 1}$:}
	\begin{equation*}
		\begin{split}
		\Vert e'_{N, 1} \Vert_{L^2(I')}^2 &\leq K_\sigma \Big( \epsilon^{-2} \Vert e_{N, 1} \Vert_{L^2(I')}^2 + \epsilon^2 \Vert e''_{N, 1} \Vert_{L^2(I')}^2 \Big) \\
		&\leq K_\sigma \Big\{ \Big( \frac{M \sigma^2}{N} \Big)^{\frac{1}{2}} \cdot 4 \gamma_N^{-1} \Vert f'' \Vert_{L^2(0, 1)}^2 \\
		&\qquad+ (4 K_\ast)^{-1} d^2 \Big( \frac{9 \beta_N + 8}{2 \gamma_N} + 36 K_\ast^2 \beta_N + 64 K_\ast^2 \Big) \Vert f'' \Vert_{L^2(0, 1)}^2 \\
		&\qquad + \Big[ \Big( \frac{M \sigma^2}{N} \Big)^{\frac{1}{2}} + 4 K_\ast d^2 \Big] \Big(\frac{4}{9} \beta_N + 4\Big) \Vert f'' \Vert_{L^2(0, 1)}^2 \Big\} \\
		&= \Big( \frac{M \sigma^2}{N} \Big)^{\frac{1}{2}} \cdot K_\sigma \Big( 4 \gamma_N^{-1} + \frac{9}{4} \beta_N + 4 \Big) \Vert f'' \Vert_{L^2(0, 1)}^2 \\
		&\qquad+ d^2 \Big( \frac{K_\sigma (9 \beta_N + 8)}{8 K_\ast \gamma_N} + 18 K_\sigma K_\ast \beta_N + 32 K_\sigma K_\ast \Big) \Vert f'' \Vert_{L^2(0, 1)}^2.
		\end{split}
	\end{equation*}
	\blue{The same approach also applies to the estimate of $f'_{N, 2}$,}
	\begin{equation*}
		\begin{split}
		\Vert f'_{N, 2} \Vert_{L^2(I')}^2 &\leq K_\sigma \Big( \epsilon^{-2} \Vert f_{N, 2} \Vert_{L^2(I')}^2 + \epsilon^2 \Vert f''_{N, 2} \Vert_{L^2(I')}^2 \Big) \\
		&\leq K_\sigma \Big\{ \Big( \frac{M \sigma^2}{N} \Big)^{\frac{1}{2}} \cdot \frac{16}{\delta \gamma_N} + d^2 \cdot \frac{4 K_\ast}{\delta} + \Big[ \Big( \frac{M \sigma^2}{N} \Big)^{\frac{1}{2}} + 4 K_\ast d^2 \Big] \cdot \frac{1}{\delta} \Big\} \\
		&= \Big( \frac{M \sigma^2}{N} \Big)^{\frac{1}{2}} \cdot K_\sigma \Big( \frac{16}{\delta \gamma_N} + \frac{1}{\delta} \Big) + d^2 \cdot \frac{8 K_\sigma K_\ast}{\delta}.
		\end{split}
	\end{equation*}
	By triangular inequality, we have
	\begin{equation*}
		\Vert e'_N \Vert_{L^2(I')} \leq \Vert e_{N, 1} \Vert_{L^2(I')} + \Vert f_{N, 2} \Vert_{L^2(I')} \leq C_3 \Big( \frac{M \sigma^2}{N} \Big)^{\frac{1}{4}} + C_4 M^{-1}.
	\end{equation*}
\end{proof}
% subsection error_bounds_in_continuous_l2_norms (end)

\subsection{Convergence rates with randomly distributed observation points} % (fold)
\label{sub:convergence_rates_with_randomly_distributed_observation_points}

If the observation points are independent and identically distributed samples from a continuous random distribution, we show asymptotic convergence rates in probability as $N \to \infty$.

\begin{assumption}\label{assumption:random-observation-points}
	The observation points $\{x_i\}_{i = 1}^{N}$ are independent and identically distributed random variables with sample space $[0, 1]$, cumulative distribution function $F(x)$, and probability density function $\rho(x)$.
	$\rho(x)$ is continuous on $[0, 1]$.
	Moreover, observation points and noise are independent, that is, the two random sequences $\{x_i\}_{i=1}^{N}$ and $\{\eta_i\}_{i=1}^{N}$ are independent.
\end{assumption}

Under the above assumption, we show the relationship between the indicator function $\rho_N(x)$ and $\rho(x)$.
We first introduce Dvoretzky–Kiefer–Wolfowitz inequality, which generates cumulative distribution function based confidence bounds of the empirical distribution function.

\begin{lemma}[Dvoretzky–Kiefer–Wolfowitz inequality]
	Under Assumption \ref{assumption:random-observation-points}, let $F_N(x)$ be the empirical distribution function of $\{x_i\}_{i=1}^{N}$, which is given by
	\begin{equation}\label{eq:empirical-distribution-definition}
		F_N(x) = \frac{1}{N} \sum_{i=1}^{N} \mathbbm{1}_{x_i \leq x},\quad x \in [0, 1].
	\end{equation}
	Then, for arbitrary $\epsilon > 0$, \blue{the difference of $F_N(x)$ and $F(x)$ can be estimated as}
	\begin{equation}\label{eq:dkw-inequality}
		\mathbb{P} \Big( \sup_{x \in [0, 1]} \big\vert F_N(x) - F(x) \big\vert > \epsilon \Big) \leq 2 e^{-2 N \epsilon^2}.
	\end{equation}
\end{lemma}

\blue{When applying Dvoretzky–Kiefer–Wolfowitz inequality to the indicator function, the following lemma is valid.}

\begin{lemma}\label{lemma:rho-n-approximation-1}
	Under Assumption \ref{assumption:random-observation-points}, for arbitrary $\delta \in (0, 1)$, the estimate
	\begin{equation*} 
		\max_{1 \leq j \leq M} \Big\vert \frac{N_j}{N d} - \frac{1}{d} \int_{(j-1)d}^{jd} \rho(x) \,d{x} \Big\vert \leq \sqrt{\frac{2 \ln(2 / \delta)}{d^2 N}}.
	\end{equation*}
	holds with a probability of at least $1 - \delta$, \blue{where $N_j (N d)^{-1}$ is the value of indicator function $\rho_N(x)$ on subinterval $I_j$.}
\end{lemma}

\begin{proof}
By \eqref{eq:dkw-inequality}, the following estimate of $F_N(x)$
\begin{equation*}
	F(x) - \epsilon \leq F_N(x) \leq F(x) + \epsilon,\quad \epsilon = \sqrt{\frac{\ln(2 / \delta)}{2 N}}
\end{equation*}
holds with a probability of at least $1 - \delta$.
By the definition of empirical distribution function \eqref{eq:empirical-distribution-definition}, \blue{the value of $\rho_N(x)$ at each subinterval $I_j$ can be written as}
\begin{equation*}
	\frac{N_j}{N d} = \frac{1}{d} \Big[ F_N(jd) - F_N \big( (j-1) d ) \Big],\quad j = 1, 2, \cdots, M.
\end{equation*}
Combine the above two formulas, \blue{the value of $\rho_N(x)$ is bounded by}
\begin{equation*}
	\frac{F(jd) - F((j-1) d)}{d} - \frac{2 \epsilon}{d} \leq \frac{N_j}{N d} \leq \frac{F(jd) - F((j-1) d)}{d} + \frac{2 \epsilon}{d},\quad j = 1, 2, \cdots, M.
\end{equation*}
\blue{Noting that $\rho(x)$ is the derivative of $F(x)$, we get to the conclusion}
% \begin{equation*}
% 	\frac{F(jd) - F((j-1) d)}{d} = \frac{1}{d} \int_{(j-1)d}^{jd} \rho(x) \,d{x}.
% \end{equation*}
\begin{equation*}
	\max_{1 \leq j \leq M} \Big\vert \frac{N_j}{N d} - \frac{1}{d} \int_{(j-1)d}^{jd} \rho(x) \,d{x} \Big\vert \leq \sqrt{\frac{2 \ln(2 / \delta)}{d^2 N}}.
\end{equation*}
\end{proof}

\begin{lemma}\label{lemma:rho-n-approximation-by-rho}
	Under Assumption \ref{assumption:random-observation-points}, suppose that $\rho(x)$ is bounded above by
	\begin{equation*}
		\sup_{x \in [0, 1]} \rho(x) \leq \beta.
	\end{equation*}
	For nonnegative integers p, q with $0 \leq p < q \leq M$, let subinterval $I' := (pd, qd)$.
	If $\rho(x)$ is also bounded below on $I'$ by
	\begin{equation*}
		\inf_{x \in I'} \rho(x) \geq \gamma > 0,
	\end{equation*}
	then for arbitrary $\delta \in (0, 1)$, provided that the number of observations satisfies
	\begin{equation*}
		N \geq \frac{8 \ln(2 / \delta)}{\gamma^2 d^2},
	\end{equation*}
	the following upper and lower bounds of indicator function
	\begin{equation}\label{eq:relationship-rho-n-rho}
	\begin{split}
		\sup_{x \in [0, 1]} \rho_N(x) &\leq 2 \beta, \\
		\inf_{x \in I'} \rho_N(x) &\geq \frac{\gamma}{2}
	\end{split}
	\end{equation}
	hold with a probability of at least $1 - \delta$.
\end{lemma}

\begin{proof}
	When $N \geq 8 \gamma^{-2} d^{-2} \ln(2 / \delta)$, we have
	\begin{equation*}
		\sqrt{\frac{2 \ln(2 / \delta)}{d^2 N}} \leq \frac{\gamma}{2}.
	\end{equation*}
	By Lemma \ref{lemma:rho-n-approximation-1}, the following bounds of indicator function
	\begin{equation*}
		\begin{split}
		\inf_{x \in I'} \rho_N(x) &= \min_{p+1 \leq j \leq q} \frac{N_j}{N d} \geq \frac{1}{d} \min_{p+1 \leq j \leq q} \int_{(j-1) d}^{jd} \rho(x) \,d{x} - \frac{\gamma}{2} \geq \frac{\gamma}{2}, \\
		\sup_{x \in [0, 1]} \rho_N(x) &= \max_{1 \leq j \leq M} \frac{N_j}{N d} \leq \frac{1}{d} \max_{1 \leq j \leq M} \int_{(j-1) d}^{jd} \rho(x) \,d{x} + \frac{\gamma}{2} \leq \beta + \frac{\gamma}{2} \leq 2 \beta
		\end{split}
	\end{equation*}
	hold with a probability of at least $1 - \delta$.
\end{proof}

Combine Lemma \ref{lemma:rho-n-approximation-by-rho} and Theorem \ref{theorem:error-bounds-given-n}, we have the following theorem:

\begin{theorem}\label{theorem:convergence-random-obs-points}
	Under Assumption \ref{assumption:random-observation-points}, suppose that $f \in W^{2, 2}(0, 1)$, $M \geq 3$, \blue{$\rho(x)$ is bounded above by}
	\begin{equation}
		\sup_{x \in [0, 1]} \rho(x) \leq \beta.
	\end{equation}
	\blue{For nonnegative integers $p$, $q$ with $0 \leq p < q \leq M$ and $q-p \geq 2 \sqrt{K_\ast}$, let subinterval $I := (pd, qd)$.}
	Choosing the regularization parameter $\alpha_N = M \sigma^2 / N + d^4$, if $\rho(x)$ is also bounded below on $I'$ by
	\begin{equation*}
		\inf_{x \in I'} \rho(x) \geq \gamma > 0,
	\end{equation*}
	then for arbitrary $\delta \in (0, 1)$, provided that
	\begin{equation*}
		N \geq \max\left\{M, \frac{8 \ln(4 / \delta)}{\gamma^2 d^2} \right\},
	\end{equation*}
	the following estimates for the $L^2$ norms of $e_N$ and $e'_N$
	\begin{equation*}
		\begin{split}
		\Vert e_N \Vert_{L^2(I')} &\leq C'_1 \Big( \frac{M \sigma^2}{N} \Big)^{\frac{1}{2}} + C'_2 M^{-2}, \\
		\Vert e'_N \Vert_{L^2(I')} &\leq C'_3 \Big( \frac{M \sigma^2}{N} \Big)^{\frac{1}{4}} + C'_4 M^{-1}
		\end{split}
	\end{equation*}
	hold with a probability of at least $1 - \delta$, where the constants are independent of $M$ and $N$.
	\blue{In Sobolev interpolation inequality \eqref{eq:sobolev-interpolation} with $I_s = I'$ and $\epsilon_0 = \max \{\sqrt{\sigma}, (q-p) d\}$, denoting by $K_\sigma$ the Sobolev constant $K(\epsilon_0, \vert I' \vert)$, the constants can be written as}
	\begin{equation*}
		\begin{split}
		C'_1 &= 2 \sqrt{2} \gamma^{-\frac{1}{2}} \Vert f'' \Vert_{L^2(0, 1)} + \frac{8}{\sqrt{\delta \gamma}}, \\
		C'_2 &= \Vert f'' \Vert_{L^2(0, 1)} \sqrt{\frac{18 \beta + 8}{\gamma} + 72 K_\ast^2 \beta + 64 K_\ast^2} + 4 \sqrt{2} K_\ast \delta^{-\frac{1}{2}}, \\
		C'_3 &= K_\sigma^{\frac{1}{2}} \Big( \Vert f'' \Vert_{L^2(0, 1)} \sqrt{8 \gamma^{-1} + \frac{9}{2} \beta + 4} + \delta^{-\frac{1}{2}} \sqrt{64 \gamma^{-1} + 2} \Big), \\
		C'_4 &= K_\sigma^{\frac{1}{2}} \Big( \Vert f'' \Vert_{L^2(0, 1)} \sqrt{\frac{9 \beta + 4}{2 K_\ast \gamma} + 36 K_\ast \beta + 32 K_\ast} + 4 \delta^{-\frac{1}{2}} K_\ast^{\frac{1}{2}} \Big).
		\end{split}
	\end{equation*}
\end{theorem}

\begin{proof}
	By Lemma \ref{lemma:rho-n-approximation-by-rho}, \eqref{eq:relationship-rho-n-rho} holds with a probability of at least $1 - \delta / 2$.
	On the other hand, suppose that \eqref{eq:relationship-rho-n-rho} holds, by Theorem \ref{theorem:error-bounds-given-n}, the following estimates
	\begin{equation*}
		\begin{split}
		\Vert e_N \Vert_{L^2(I')} &\leq C'_1 \Big( \frac{M \sigma^2}{N} \Big)^{\frac{1}{2}} + C'_2 M^{-2}, \\
		\Vert e'_N \Vert_{L^2(I')} &\leq C'_3 \Big( \frac{M \sigma^2}{N} \Big)^{\frac{1}{4}} + C'_4 M^{-1}
		\end{split}
	\end{equation*}
	hold with a probability of at least $1 - \delta / 2$.
	Merging these two confidence intervals completes the proof.
\end{proof}

\begin{remark}[The choice of $M$]\label{remark:choice-of-m}
	Suppose that $\rho(x)$ has a positive lower bound on $[0, 1]$, we obtain optimal asymptotic convergence rates
	\begin{equation}\label{eq:convergence-rates-in-probability}
		\Vert e_N \Vert_{L^2(0, 1)} \sim O_{p}(N^{-2/5}),\quad \Vert e'_N \Vert_{L^2(0, 1)} \sim O_{p}(N^{-1/5})
	\end{equation}
	by letting $M \sim N^{1/5}$.
	Therefore, if $N$ is known in advance, one may choose $M \sim N^{1/5}$.
	On the other hand, if $N$ is unknown, one may choose $M$ according to requirements for accuracy, since the reconstruction accuracy is roughly the same as the cubic spline interpolation with knots $\{p_j\}_{j=0}^{M}$ when $N$ is sufficiently large.
\end{remark}

% subsection convergence_rates_with_randomly_distributed_observation_points (end)
% section error_analysis (end)

\section{Numerical examples} % (fold)
\label{sec:numerical_examples}

In this section we give some numerical examples.

\subsection{Numerical results with different distributions of observation points} % (fold)
\label{sub:numerical_results_with_different_distributions_of_observation_points}

Let the unknown function
\begin{equation*}
    f(x) = \frac{1}{100} \Big[ x^2 + 3 x + \sin 4 \pi x + 2 \exp\big( -8 (x - 2/5)^2 \big) \Big],
\end{equation*}
and
\begin{equation*}
    \sigma^2 = 5 \times 10^{-5},\quad N = 600,\quad M = 40,\quad \alpha_N = M \sigma^2 / N + d^4 = 3.72 \times 10^{-6}.
\end{equation*}
To simulate the real situations where observation points are unevenly distributed, we consider the following three types of distributions:
\begin{enumerate}
    \item Uniform: Observation points are uniformly distributed on $[0, 1]$.
    \item Concentrated on the left side: Observation points are uniformly distributed on $[0, 1/2)$ and $[1/2, 1]$ with probabilities of $0.95$ and $0.05$, respectively.
    \item Concentrated on both ends: Observation points are uniformly distributed on $[0, 1/5)$, $[1/5, 4/5]$, and $(4/5, 1]$ with probabilities of $0.475$, $0.05$, and $0.475$, respectively.
\end{enumerate}

Figures \ref{fig:complicated-uniform}, \ref{fig:complicated-left}, and \ref{fig:complicated-middle} are the reconstruction results of $f(x)$ with those three distributions of observation points.
Reconstructed functions and derivatives are shown in blue solid lines, and corresponding truths are shown in black dashed lines.
In Figure \ref{fig:complicated-uniform} where observation points are uniformly distributed, the reconstructed results are satisfactory.
In Figures \ref{fig:complicated-left} and \ref{fig:complicated-middle}, the results become inaccurate at locations where observation points are sparse.
It is notable that the corresponding histograms effectively indicates the quality of reconstruction results in different regions, i.e., reconstruction results are more accurate in regions with higher histogram bars, and vice versa.

\begin{figure}[htbp]
    \includegraphics[width=0.9\linewidth]{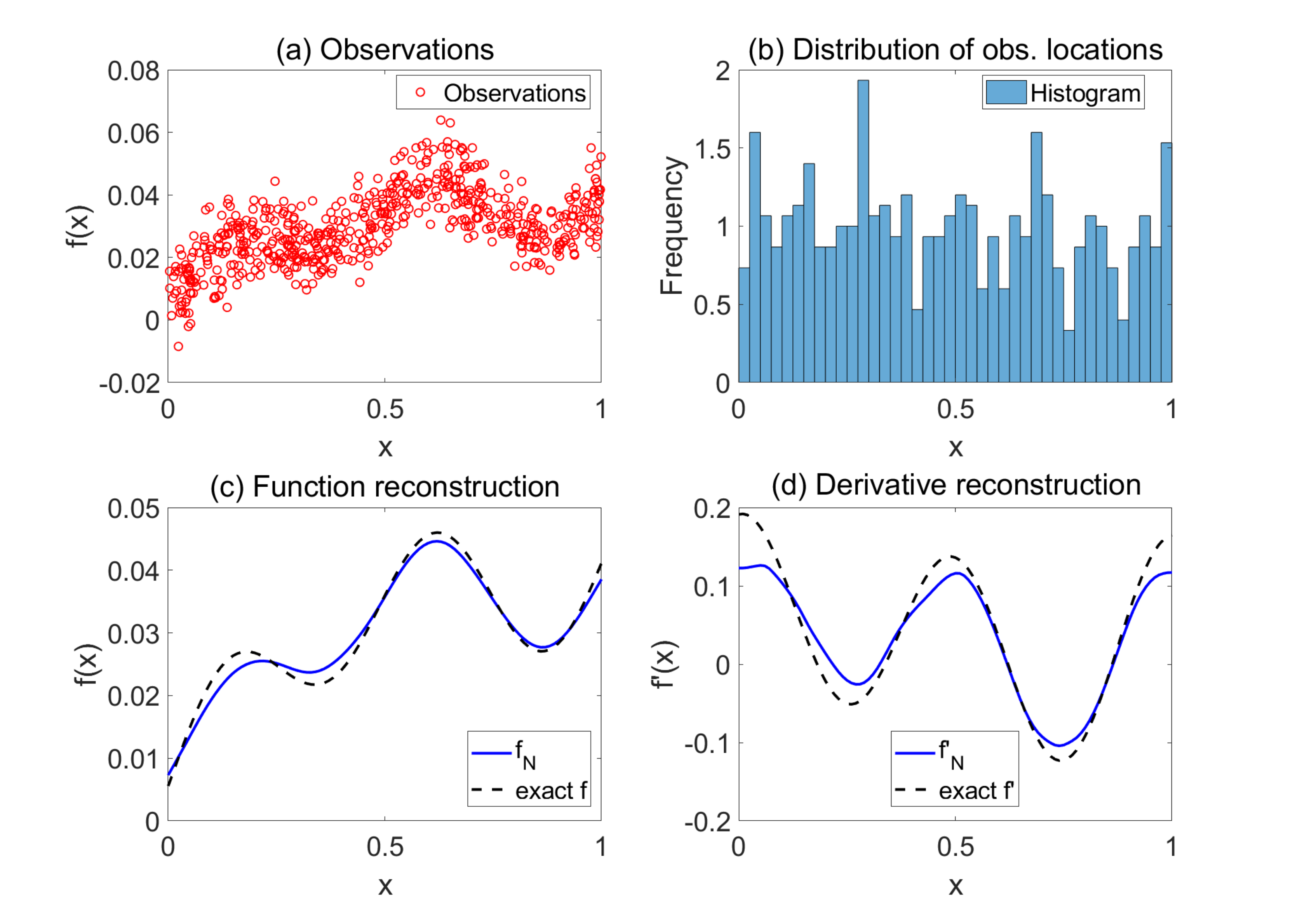}
    \caption{Reconstruction results of a more complicated function, observation points are uniformly distributed.}
    \label{fig:complicated-uniform}
\end{figure}

\begin{figure}[htbp]
    \includegraphics[width=0.9\linewidth]{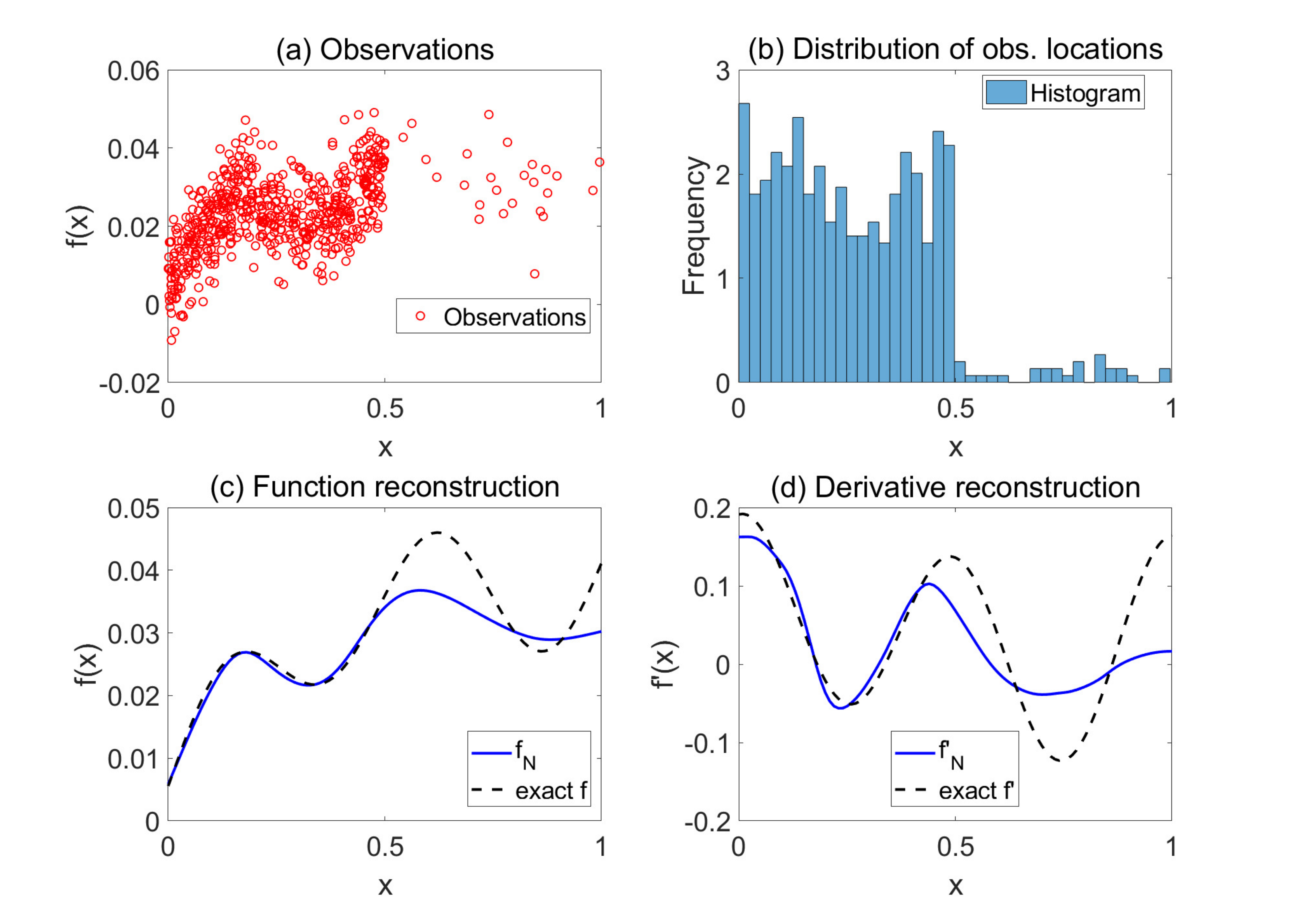}
    \caption{Reconstruction results of a more complicated function, observation points are concentrated on the left side.}
    \label{fig:complicated-left}
\end{figure}

\begin{figure}[htbp]
    \includegraphics[width=0.9\linewidth]{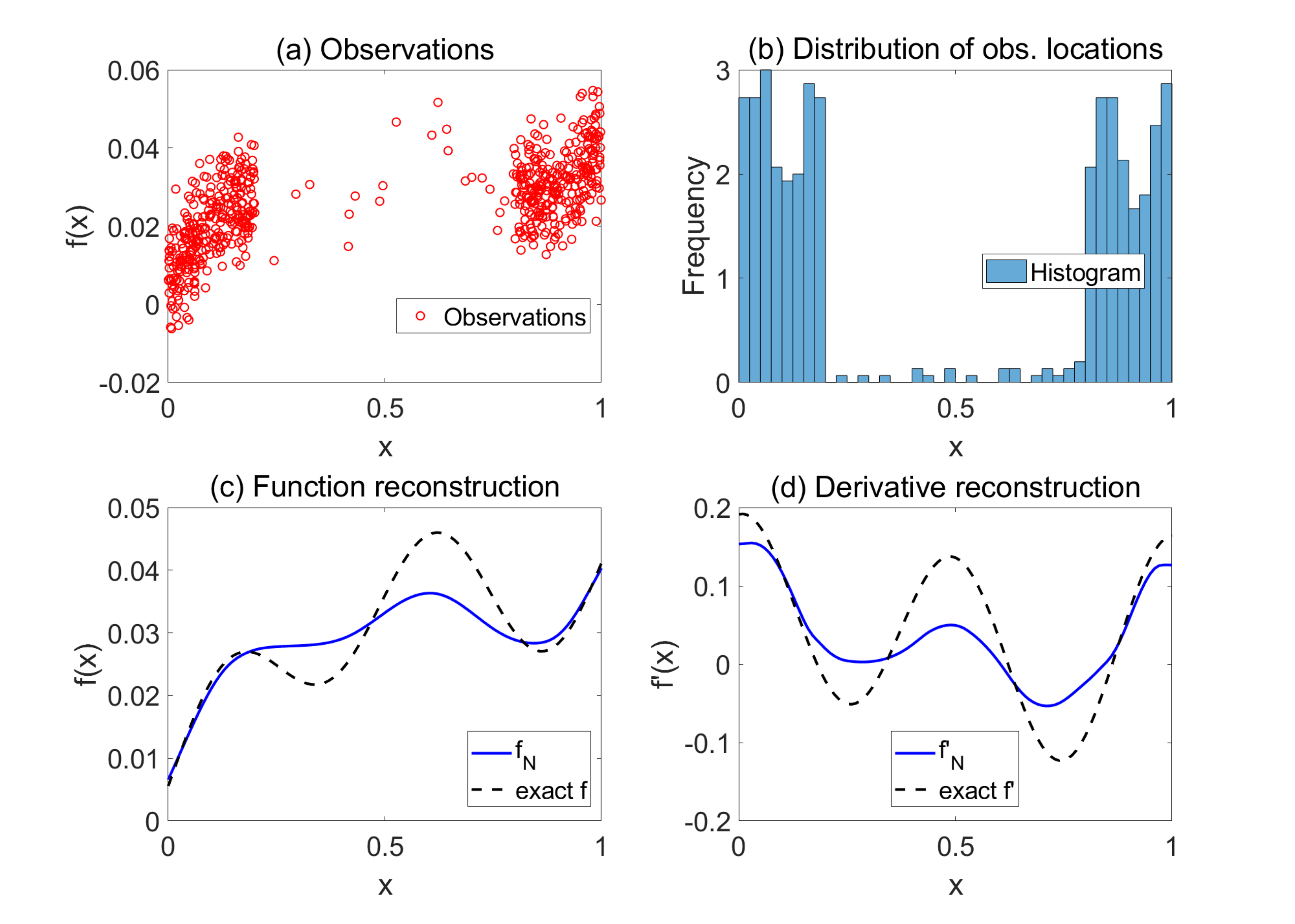}
    \caption{Reconstruction results of a more complicated function, observation points are concentrated on both ends.}
    \label{fig:complicated-middle}
\end{figure}

\newpage

\subsection{Asymptotic convergence rates} % (fold)
\label{sub:asymptotic_convergence_rates}

We design a numerical test to verify the asymptotic convergence rates of $\Vert e_N \Vert_{L^2(0, 1)} \sim N^{-2/5}$, $\Vert e'_N \Vert_{L^2(0, 1)} \sim N^{-1/5}$.
Let
\begin{equation*}
	f(x) = \frac{1}{100} \Big[ x^2 + 3 x + \sin 4 \pi x + 2 \exp\big( -8 (x - 2/5)^2 \big) \Big],\quad \sigma^2 = 1 \times 10^{-4}.
\end{equation*}
We vary the value of $M$ from $50$ to $250$ and let $N = M^5 / 10000$.
More precisely, the pairs of values of $N$ and $M$ that are used in this test can be found in the following table.

\begin{center}
\begin{tabular}{|c|c|c|c|c|c|c|c|}
	\hline
	$N$ & $31250$ & $77760$ & $168070$ & $327680$ & $590490$ & $1000000$ & $1610510$ \\
	\hline
	$M$ & $50$ & $60$ & $70$ & $80$ & $90$ & $100$ & $110$ \\
	\hline
	$N$ & $2488320$ & $3712930$ & $5378240$ & $7593750$ & $10485760$ & $14198570$ & $18895680$ \\
	\hline
	$M$ & $120$ & $130$ & $140$ & $150$ & $160$ & $170$ & $180$ \\
	\hline
	$N$ & $24760990$ & $32000000$ & $40841010$ & $51536320$ & $64363430$ & $79626240$ & $97656250$ \\
	\hline
	$M$ & $190$ & $200$ & $210$ & $220$ & $230$ & $240$ & $250$ \\
	\hline
\end{tabular}
\end{center}

For each pair of $N$ and $M$, we generate observation points that are uniformly distributed on $[0, 1]$, and then compute the $L^2$ errors of $\Vert e_N \Vert_{L^2(0, 1)}$ and $\Vert e'_N \Vert_{L^2(0, 1)}$.
We run this process for $12$ times, and take averages of those $L^2$ errors that are obtained in every run.
Figure \ref{fig:convergence-rates} verifies that the convergence rates are $O_p(N^{-2/5})$ for function reconstruction, and $O_p(N^{-1/5})$ for derivative reconstruction.

\begin{figure}[htbp]
	\includegraphics[width=1\linewidth]{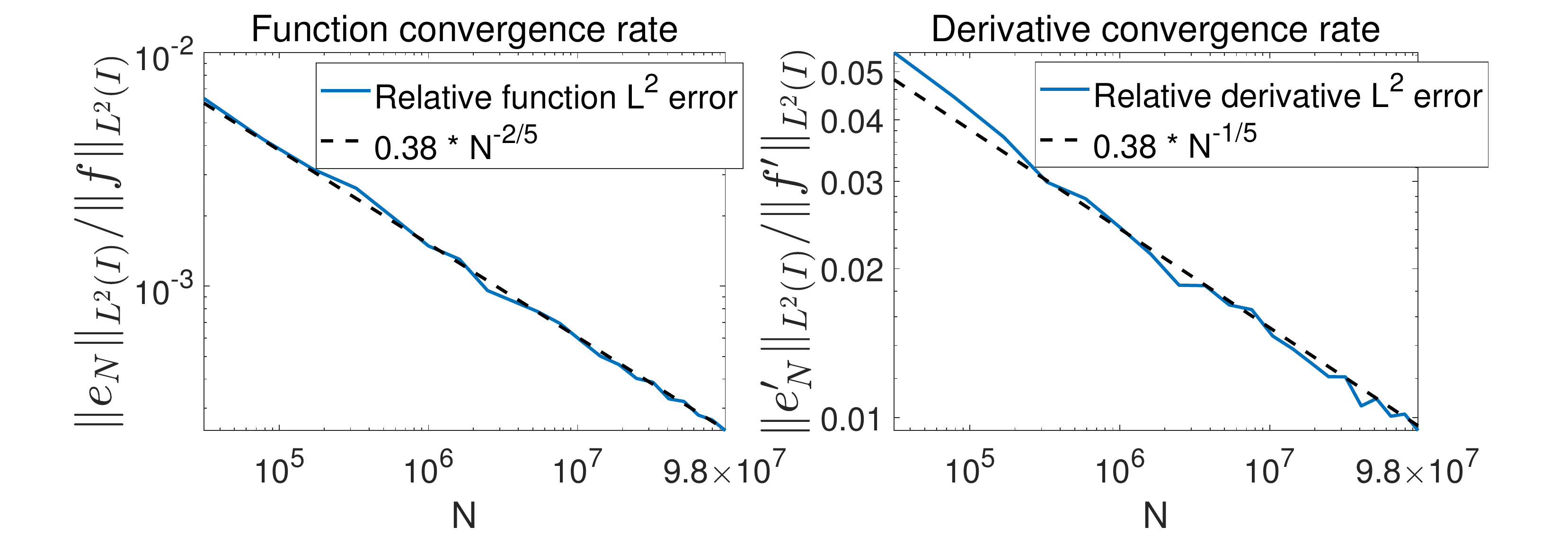}
	\caption{Convergence rates of function (left) and derivative (right) reconstructions.}
	\label{fig:convergence-rates}
\end{figure}

\newpage
% subsection asymptotic_convergence_rates (end)
% section numerical_examples (end)

\section{Conclusion} % (fold)
\label{sec:conclusion}

In this paper, we proposed a big data processing technique based on Tikhonov regularization for one-dimensional data with random noise.
The algorithm can process large datasets with low computational cost.
For error analysis, we proposed an indicator function that shows reliable regions with a given dataset.
Provided that observation points are randomly distributed, we also derived optimal asymptotic convergence rates with respect to sample size.
This technique can be used as a preprocessing method in numerical methods for inverse problems.

% section conclusion (end)

\appendix
\section{Cubic B-splines} % (fold)
\label{sec:appendix-a-b-splines}

We introduce cubic B-spline function to construct a basis for $V_M$, defined in Definition \ref{definition:vm}.
Readers may refer to \cite{Micula.1999} for details.

The cubic B-spline function $\phi_3(x)$ is given by
\begin{equation*}
	\phi_3(x) = \frac{1}{6} \times \begin{cases}
		(x+2)^3, &\text{ if }-2 \leq x \leq -1, \\
		(x+2)^3 - 4 (x+1)^3, &\text{ if }-1 < x \leq 0, \\
		(2-x)^3 - 4(1-x)^3, &\text{ if }0 < x \leq 1, \\
		(2-x)^3, &\text{ if }1 < x \leq 2, \\
		0, &\text{ if }\vert x \vert > 2.
	\end{cases}
\end{equation*}
Here, $\phi_3(x)$ is a piecewise polynomial of order $3$ that belongs to $C^2(\mathbb{R})$, and has a compact support of $[-2, 2]$.
By translation and scaling, $\phi_3(x)$ forms a basis of $V_M$.

\begin{proposition}[{\cite[Theorem~1.8]{Micula.1999}}]
	Let
	\begin{equation}\label{eq:psi-definition}
		\psi_j = \phi_3 \Big( \frac{x - p_j}{d} \Big),\quad -1 \leq j \leq M+1,
	\end{equation}
	then $\{\psi_j\}_{j = -1}^{M+1}$ forms a basis of the $(M+3)$-dimensional vector space $V_M$.
\end{proposition}

\FloatBarrier
\bibliographystyle{amsplain}
\bibliography{references}

\providecommand{\bysame}{\leavevmode\hbox to3em{\hrulefill}\thinspace}
\providecommand{\MR}{\relax\ifhmode\unskip\space\fi MR }
\providecommand{\MRhref}[2]{%
  \href{http://www.ams.org/mathscinet-getitem?mr=#1}{#2}
}
\providecommand{\href}[2]{#2}
\begin{thebibliography}{10}

\bibitem{Adams.2003}
R.~A. Adams and John J.~F. Fournier, \emph{Sobolev spaces}, 2nd ed. ed., Pure
  and applied mathematics, vol. v. 140, {Academic Press}, Amsterdam and Boston,
  2003.

\bibitem{Ahnert.2007}
Karsten Ahnert and Markus Abel, \emph{Numerical differentiation of experimental
  data: local versus global methods}, Computer Physics Communications
  \textbf{177} (2007), no.~10, 764--774.

\bibitem{Cheng.2007}
J.~Cheng, X.~Z. Jia, and Y.~B. Wang, \emph{Numerical differentiation and its
  applications}, Inverse Problems in Science and Engineering \textbf{15}
  (2007), no.~4, 339--357.

\bibitem{Claeskens.2009}
G.~Claeskens, T.~Krivobokova, and J.~D. Opsomer, \emph{Asymptotic properties of
  penalized spline estimators}, Biometrika \textbf{96} (2009), no.~3, 529--544.

\bibitem{Craven.1978}
Peter Craven and Grace Wahba, \emph{Smoothing noisy data with spline
  functions}, Numerische Mathematik \textbf{31} (1978), no.~4, 377--403.

\bibitem{Deans.2007}
Stanley~R. Deans, \emph{The radon transform and some of its applications},
  {Dover Publications}, Mineola N.Y., 2007.

\bibitem{Eilers.1996}
Paul H.~C. Eilers and Brian~D. Marx, \emph{Flexible smoothing with b -splines
  and penalties}, Statistical Science \textbf{11} (1996), no.~2, 89--121.

\bibitem{Gorenflo.2006}
Rudolf Gorenflo and Sergio Vessella, \emph{Abel integral equations: Analysis
  and applications}, Lecture Notes in Mathematics Ser, vol. v. 1461, {Springer
  Berlin / Heidelberg}, Berlin, Heidelberg, 2006.

\bibitem{Hanke.2001}
Martin Hanke and Otmar Scherzer, \emph{Inverse problems light: Numerical
  differentiation}, The American Mathematical Monthly \textbf{108} (2001),
  no.~6, 512--521.

\bibitem{Hu.2012}
Bang Hu and Shuai Lu, \emph{Numerical differentiation by a tikhonov
  regularization method based on the discrete cosine transform}, Applicable
  Analysis \textbf{91} (2012), no.~4, 719--736.

\bibitem{Lu.2006b}
Shuai Lu and Sergei~V. Pereverzev, \emph{Numerical differentiation from a
  viewpoint of regularization theory}, Mathematics of Computation \textbf{75}
  (2006), no.~256, 1853--1870.

\bibitem{Lu.2006}
Shuai Lu and Yanbo Wang, \emph{First and second order numerical differentiation
  with tikhonov regularization}, Frontiers of Mathematics in China \textbf{1}
  (2006), no.~3, 354--367.

\bibitem{Micula.1999}
Gheorghe Micula, \emph{Handbook of splines}, Mathematics and Its Applications,
  vol. 462, {Springer Netherlands}, Dordrecht, 1999.

\bibitem{OSullivan.1986}
Finbarr O'Sullivan, \emph{A statistical perspective on ill-posed inverse
  problems}, Statistical Science \textbf{1} (1986), no.~4, 502--518.

\bibitem{Ragozin.1983}
David~L. Ragozin, \emph{Error bounds for derivative estimates based on spline
  smoothing of exact or noisy data}, Journal of Approximation Theory
  \textbf{37} (1983), no.~4, 335--355.

\bibitem{Scott.1989}
Larkin~B. Scott and L.~Ridgway Scott, \emph{Efficient methods for data
  smoothing}, SIAM Journal on Numerical Analysis \textbf{26} (1989), no.~3,
  681--692.

\bibitem{Wahba.1975}
Grace Wahba, \emph{Smoothing noisy data with spline functions}, Numerische
  Mathematik \textbf{24} (1975), no.~5, 383--393.

\bibitem{Wand.2008}
M.~P. Wand and J.~T. Ormerod, \emph{On semiparametric regression with
  o'sullivan penalized splines}, Australian {\&} New Zealand Journal of
  Statistics \textbf{50} (2008), no.~2, 179--198.

\bibitem{YBWang.2002}
{Y B Wang}, {X Z Jia}, and {J Cheng}, \emph{A numerical differentiation method
  and its application to reconstruction of discontinuity}, Inverse Problems
  \textbf{18} (2002), no.~6, 1461.

\end{thebibliography}

\end{document}